\documentclass[a4paper,10pt]{article}
\usepackage[ansinew]{inputenc}          
\usepackage{amssymb,amsthm,amsmath}     
\usepackage[dvips]{graphicx}            
\usepackage[all]{xy}
\usepackage{subfigure}
\usepackage{algorithm,algorithmic}
\usepackage[hmargin=3.7cm,vmargin=4cm]{geometry}

\theoremstyle{plain}
\newtheorem{thm}{Theorem}[section]

\newtheorem{lemma}[thm]{Lemma}

\theoremstyle{definition}
\newtheorem{definition}[thm]{Definition}
\newtheorem{example}[thm]{Example}

\newcommand{\Z}{\mathbb{Z}}

\newcommand{\St}{\mathcal{S}}

\newcommand{\bu}{\mathbf{u}}
\newcommand{\bv}{\mathbf{v}}
\newcommand{\bc}{\mathbf{c}}

\newcommand{\ms}{\bar{s}}

\begin{document}

\title{New lower bound for $2$-identifying code in the square grid}

\author{\textbf{Ville Junnila}\thanks{Phone +358 2 333 6021, Fax +358 2 231 0311.} \\
Department of Mathematics\\
University of Turku, FI-20014 Turku, Finland\\
viljun@utu.fi}
\date{}
\maketitle

\begin{abstract}
An $r$-identifying code in a graph $G = (V,E)$ is a subset $C \subseteq V$ such that for each $u \in V$ the intersection of $C$ and the ball of radius $r$ centered at $u$ is nonempty and unique. Previously, $r$-identifying codes have been studied in various grids. In particular, it has been shown that there exists a $2$-identifying code in the square grid with density $5/29 \approx 0.172$ and that there are no $2$-identifying codes with density smaller than $3/20 = 0.15$. Recently, the lower bound has been improved to $6/37 \approx 0.162$ by Martin and Stanton (2010). In this paper, we further improve the lower bound by showing that there are no $2$-identifying codes in the square grid with density smaller than $6/35 \approx 0.171$.
\end{abstract}


\section{Introduction}

Let $G = (V, E)$ be a simple, connected and undirected graph with $V$ as the set of vertices and $E$ as the set of edges. Let $u$ and $v$ be vertices in $V$. If $u$ and $v$ are adjacent to each other, then the edge between $u$ and $v$ is denoted by $\{u, v\}$ or in short by $uv$. The \emph{distance} between $u$ and $v$ is denoted by $d(u,v)$ and is defined as the number of edges in any shortest path between $u$ and $v$. Let $r$ be a positive integer. We say that $u$ $r$\emph{-covers} $v$ if the distance $d(u,v)$ is at most $r$. The \emph{ball of radius} $r$ \emph{centered at} $u$ is defined as
\[
B_r(u) = \{ x \in V \ | \ d(u,x) \leq r \} \textrm{.}
\]

A nonempty subset of $V$ is called a \emph{code} in $G$, and its elements are called \emph{codewords}. Let $C \subseteq V$ be a code and $u$ be a vertex in $V$. An \emph{I-set} (or an \emph{identifying set}) of the vertex $u$ with respect to the code $C$ is defined as \[
I_r(C;u) = I_r(u) = B_r(u) \cap C \textrm{.}
\]
The following definition of identifying codes is due to Karpovsky \emph{et al.} \cite{kcl}.
\begin{definition}
Let $r$ be a positive integer. A code $C \subseteq V$ is said to be $r$-\emph{identifying} in $G$ if for all $u,v \in V$ ($u \neq v$) the set $I_r(C;u)$ is nonempty and
\[
I_r(C;u) \neq I_r(C;v) \textrm{.}
\]
\end{definition}

Let $X$ and $Y$ be subsets of $V$. The \emph{symmetric difference} of $X$ and $Y$ is defined as $X \, \triangle \, Y = (X \setminus Y) \cup (Y \setminus X)$. We say that the vertices $u$ and $v$ are $r$\emph{-separated} by a code $C \subseteq V$ (or by a codeword of $C$) if the symmetric difference $I_r(C;u) \, \triangle \, I_r(C;v)$ is nonempty. The definition of $r$-identifying codes can now be reformulated as follows: $C \subseteq V$ is an $r$-identifying code in $G$ if and only if for all $u,v \in V$ ($u \neq v$) the vertex $u$ is $r$-covered by a codeword of $C$ and 
\[
I_r(C;u) \, \triangle \, I_r(C;v) \neq \emptyset \textrm{.}
\]

In this paper, we study identifying codes in the square grid. We define the square grid $G_S = (V_S, E_S)$ as follows: the set of vertices $V_S = \Z^2$ and the set of edges
\[
E_S = \{ \{ \bu, \bv \} \ | \ \bu, \bv \in \Z^2, \bu - \bv \in \{(0,\pm 1), (\pm 1, 0)\} \} \textrm{.}
\]
In other words, two vertices are adjacent in $G_S$ if the Euclidean distance between them is equal to $1$. Part of the infinite square grid $G_S$ is illustrated in Figure~\ref{SquareIllustrated}, where lines represent the edges and intersections of the lines represent the vertices of $G_S$. A $2$-identifying code in $G_S$, which is formed by the shaded vertices, is constructed by repeating the pattern in the dashed box.

\begin{figure}
\centering
\includegraphics[height=175pt]{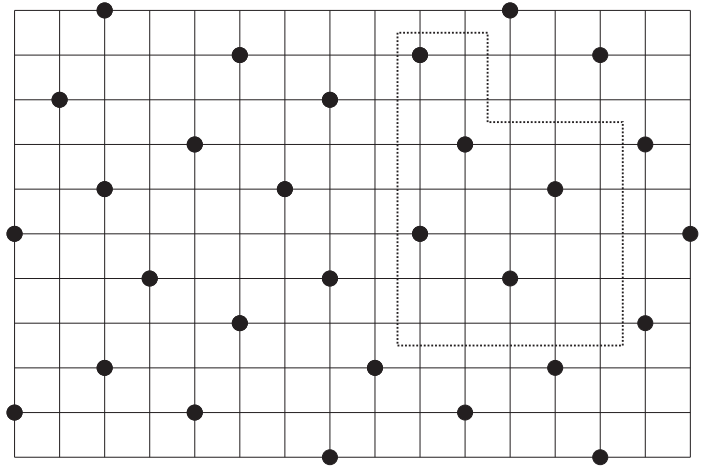}
\caption{The square grid $G_S$ and a $2$-identifying code in $G_S$ with density $5/29$ illustrated.} 
\label{SquareIllustrated}
\end{figure}

To measure the size of an identifying code in the infinite square grid, we introduce the notion of density. For the formal definition, we first define
\[
Q_n = \{ (x,y) \in V_S \ | \ |x| \leq n, |y| \leq n \} \textrm{.}
\]
Then the \emph{density} of a code $C \subseteq V_S$ is defined as
\[
D(C) = \limsup_{n \rightarrow \infty} \frac{|C \cap Q_n|}{|Q_n|} \textrm{.}
\]
Naturally, we try to construct identifying codes with as small density as possible. Moreover, we say that an $r$-identifying code is \emph{optimal}, if there do not exist any $r$-identifying codes with smaller density.

Previously, $r$-identifying codes in $G_S$ have been studied in various papers. In the case $r=1$, Cohen \emph{et al.} \cite{monta} constructed a $1$-identifying code in $G_S$ with density $7/20$. Moreover, it has been proved by Ben-Haim and Litsyn \cite{BHL:ExtMinDenSqr} that this construction is optimal, i.e. there do not exist $1$-identifying codes in $G_S$ with smaller density. For general $r \geq 2$, Charon \emph{et al.} \cite{chhl} showed that each $r$-identifying code $C$ in $G_S$ has density $D(C) \geq 3/(8r+4)$. Furthermore, Honkala and Lobstein \cite{hlo} presented $r$-identifying codes in the square with $D(C) = 2/(5r)$ if $r$ is even and $D(C) = 2r/(5r^2-2r+1)$ if $r$ is odd. For small values of $r$, these general constructions have been improved in \cite{chl} by Charon \emph{et al}.

In this paper, we focus our study to the case $r=2$. In \cite{hlo}, besides the general constructions, Honkala and Lobstein also presented a $2$-identifying code with density $5/29 \approx 0.172$ (see Figure~\ref{SquareIllustrated}). By the general lower bound mentioned above, we know that each $2$-identifying code $C$ in $G_S$ satisfies $D(C) \geq 3/20 = 0.15$. This general lower bound was improved by Martin and Stanton in \cite{MSlbg} by showing that the density of any $2$-identifying code in the square grid is at least $6/37 \approx 0.162$. In this paper, we further improve this lower bound to $6/35 \approx 0.171$ hence significantly reducing the gap between the lower and upper bounds.

The proof of the lower bound is based on a technique which combines the concept of share with an averaging process. Previously, a similar approach has been used in \cite{JLoidhexWCC} and \cite{JLolhg} to get an optimal lower bound for the $2$-identifying code in the hexagonal grid. The share and its usage in obtaining lower bounds are explained in Section~\ref{SquareBasics}, and the actual proof with the averaging process is presented in Section~\ref{SquareLowerBound}. Although the technique applied in this paper is similar to the one previously used in the case of hexagonal grid, things get much more complicated in this case. Therefore, in order to prove the lower bound, we need to combine exact mathematical proofs with some exhaustive computer searches as explained in Section~\ref{SquareLowerBound}.

\section{Lower bounds using share} \label{SquareBasics}

Let $G = (V, E)$ be a simple, connected and undirected graph. Assume that $C$ is a code in $G$. The following concept of the share of a codeword has been introduced by Slater in \cite{S:fault-tolerant}. The \emph{share} of a codeword $c \in C$ is defined as
\[
s_r(C; c) = s_r(c) = \sum_{u \in B_r(c)} \frac{1}{|I_r(C; u)|} \textrm{.}
\]
The notion of share proves to be useful in determining lower bounds of $r$-identifying codes (as explained in the following).

Assume that $G=(V,E)$ is a finite graph and $C$ is a code in $G$ such that $B_r(u) \cap C$ is nonempty for all $u \in V$. Then it is easy to conclude that $\sum_{c \in C} s_r(C;c) = |V|$. Assume further that $s_r(C;c) \leq \alpha$ for all $c \in C$. Then we have $|V| \leq \alpha |C|$, which immediately implies
\[
|C| \geq \frac{1}{\alpha} |V| \textrm{.}
\]
Assume then that for any $r$-identifying code $C$ in $G$ we have $s_r(C;c) \leq \alpha$ for all $c \in C$. By the aforementioned observation, we then obtain the lower bound $|V|/\alpha$ for the size of an $r$-identifying code in $G$. In other words, by determining the maximum share, we obtain a lower bound for the minimum size of an $r$-identifying code.

The previous reasoning can also be generalized to the case when an infinite graph is considered. In particular, if for any $r$-identifying code in $G_S$ we have $s_r(C;\bc) \leq \alpha$ for all $\bc \in C$, then it can be shown that the density of an $r$-identifying code in $G_S$ is at least $1/\alpha$ (compare to Theorem~\ref{SquareR2MainThm}). The main idea behind the proof of the lower bound (in Section~\ref{SquareLowerBound}) is based on this observation, although we use a more sophisticated method by showing that for any $2$-identifying code the share is on \emph{average} at most $35/6$. In Theorem~\ref{SquareR2MainThm}, we present a formal proof to verify that this method is indeed valid.

\medskip

In the proof of the lower bound, we need to determine upper bounds for shares of codewords. 
To formally present a way to estimate shares, we first need to introduce some notation. Let $C$ be an $r$-identifying code in $G$, $C'$ be a subset of $C$ and $c$ be a codeword belonging to $C'$. Since $C$ is an $r$-identifying code in $G$, the identifying sets $I_r(C;u)$ are nonempty and unique for all $u \in B_r(c)$. However, as $C'$ is a subset of $C$, all the $I$-sets $I_r(C';u)$ are not necessarily different (when $u$ goes through the vertices in $B_r(c)$). Assume that among these $I$-sets there exists $k$ different ones and that these different identifying sets are denoted by $I_1, I_2, \ldots, I_k$. Furthermore, denote the number of identifying sets equal to $I_j$ by $i_j$ $(j=1,2,\ldots,k)$. Now we are ready to present the following lemma, which provides a method to estimate the shares of the codewords.
\begin{lemma} \label{SimpleEstLemma}
Let $C$ be an $r$-identifying code in $G$ and let $C'$ be a nonempty subset of $C$. For $c \in C'$, using the previous notations, we have
\[
s_r(C;c) \leq \sum_{j=1}^{k} \left( \frac{1}{|I_j|} + (i_j-1)\frac{1}{|I_j|+1} \right) \textrm{.}
\]
\end{lemma}
\begin{proof}
Assume that $c \in C'$. Then, for each $j=1,2,\ldots,k$, define $\mathcal{I}_j = \{ u \in B_r(c) \, | \, I_j=I_r(C';u) \}$. 
Now it is obvious that for at most one vertex $u \in \mathcal{I}_j$ we have $I_j = I_r(C;u)$ and the other vertices of $\mathcal{I}_j$ are $r$-covered by at least $|I_j|+1$ codewords of $C$. Hence, the claim immediately follows.
\end{proof}

The previous lemma will be used numerous times in this paper. The computations needed in applying this lemma may sometimes be a little bit tedious, but always very straightforward. Moreover, it is easy to implement an algorithm to compute the upper bound given by the lemma. We illustrate the use of the previous lemma in the following example.

\begin{figure}
\centering
\includegraphics[height=170pt]{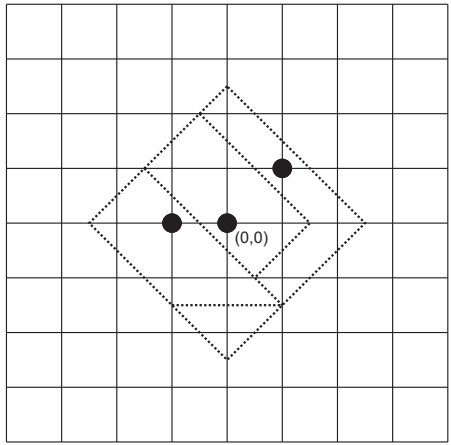}
\caption{The Example \ref{SquareIllExampleAdjacentLemma} illustrated. The shaded dots represent codewords of $C$.} \label{SquareExampleIllustrated}
\end{figure}

\begin{example} \label{SquareIllExampleAdjacentLemma}
Let $C$ be a $2$-identifying code in the square grid $G_S$. Assume that $C' = \{(-1,0),(0,0),(1,1)\}$ is a subset of $C$ (see Figure~\ref{SquareExampleIllustrated}). Now we have the following facts (which correspond to the dashed areas in the figure):
\begin{itemize}
\item $I_2(C';\bu) = \{(0,0),(1,1)\}$ for $\bu \in \{(0,2), (1,-1), (1,1),(2,0)\}$,
\item $I_2(C';\bu) = \{(-1,0),(0,0),(1,1)\}$ for $\bu \in \{(-1,1), (0,0), (0,1),(1,0)\}$,
\item $I_2(C';\bu) = \{(-1,0),(0,0)\}$ for $\bu \in \{(-2,0), (-1,-1), (-1,0),(0,-1)\}$, and
\item $I_2(C';(0,-2)) = \{(0,0)\}$.
\end{itemize}
Thus, by Lemma~\ref{SimpleEstLemma}, we obtain that
\[
s_2(C;(0,0)) \leq \left(\frac{1}{2} + 3 \cdot \frac{1}{3} \right) + \left(\frac{1}{3} + 3 \cdot \frac{1}{4} \right) + \left(\frac{1}{2} + 3\cdot \frac{1}{3} \right) + 1 = \frac{61}{12} \textrm{.}
\]
\end{example}

\section{The proof of the lower bound} \label{SquareLowerBound}

In this section, we assume that $C$ is a $2$-identifying code in $G_S$. In what follows, we show that on average the share of a codeword is at most $35/6$. Therefore, as shown in Theorem~\ref{SquareR2MainThm}, we obtain that the density $D(C) \geq 6/35$.

The averaging process is done by introducing a \emph{shifting scheme} designed to even out the shares among the codewords of $C$. The shifting scheme can also be understood as a discharging method, which is a terminology more commonly used in the literature. The rules of the shifting scheme are defined in Section~\ref{SquareSubsectionRules}. In Section~\ref{SquareSubsectionMainTheorem}, we introduce three lemmas, which state the following results:
\begin{itemize}
\item If $s_2(\bc) > 35/6$ for some $\bc \in C$, then at least $s_2(\bc) - 35/6$ units of share is shifted from $\bc$ to other codewords. (Lemma~\ref{SquareR2ReceivingLemma})
\item If share is shifted to a codeword $\bc \in C$, then $s_2(\bc) \leq 35/6$ and the codeword $\bc$ receives at most $35/6 - s_2(\bc)$ units of share. (Lemmas~\ref{SquareR2ReceiveLemma2} and \ref{SquareR2ReceiveLemma1})
\end{itemize}
In other words, after the shifting is done, the share of each codeword is at most $35/6$. Using this fact, we are able to prove the main theorem (Theorem~\ref{SquareR2MainThm}) of the paper according to which $D(C) \geq 6/35$. Finally, in Section~\ref{SquareSubsectionLemmasProofs}, we provide the proofs of the lemmas.

\subsection{The rules of the shifting scheme} \label{SquareSubsectionRules}

\begin{figure}[htp]
\centering
\subfigure[Rule~1]{
\includegraphics[height=100pt]{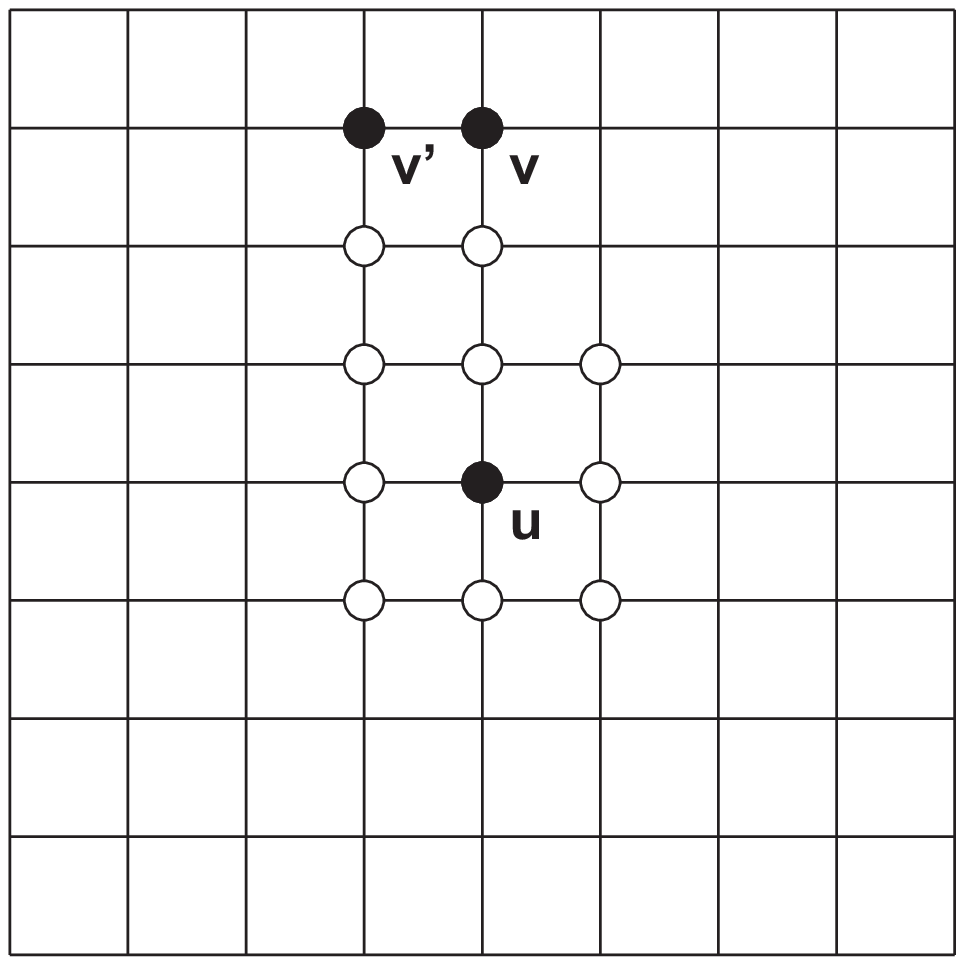}
\label{SquareR2Rule1}
}
\subfigure[Rule~2]{
\includegraphics[height=100pt]{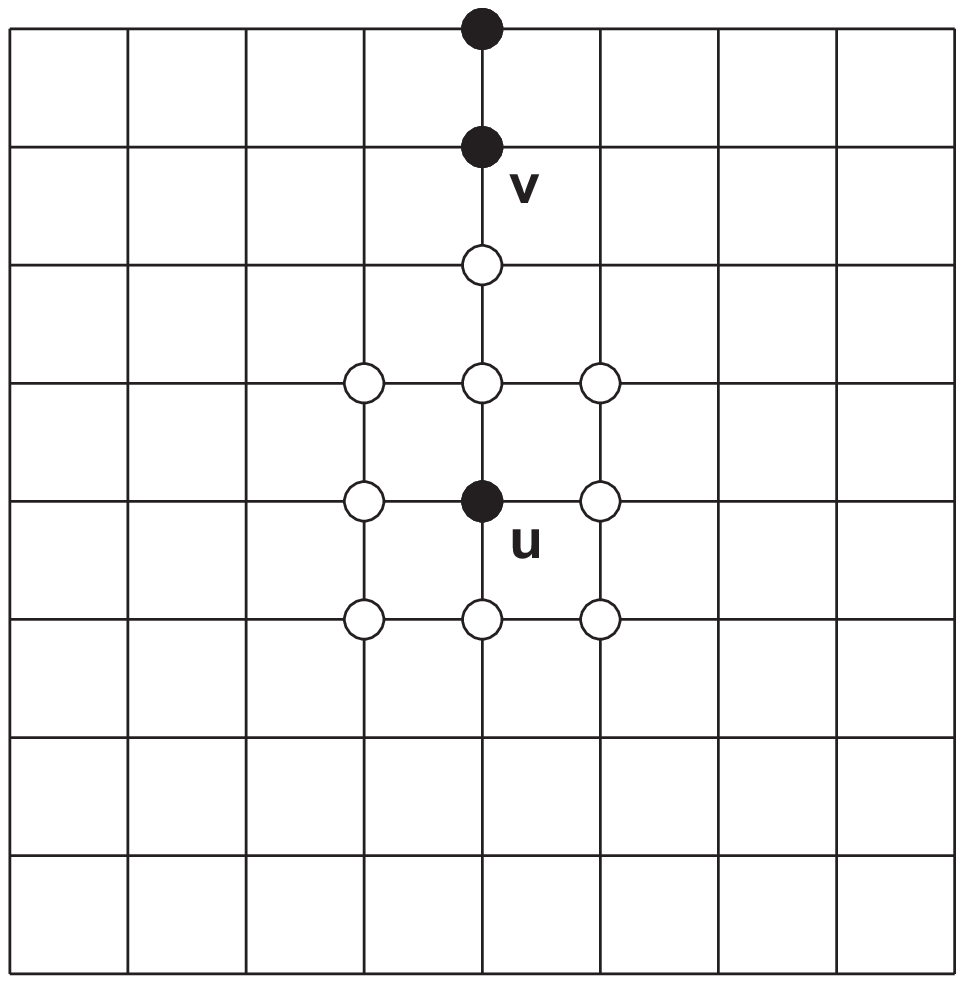}
\label{SquareR2Rule2}
}
\subfigure[Rule~3]{
\includegraphics[height=100pt]{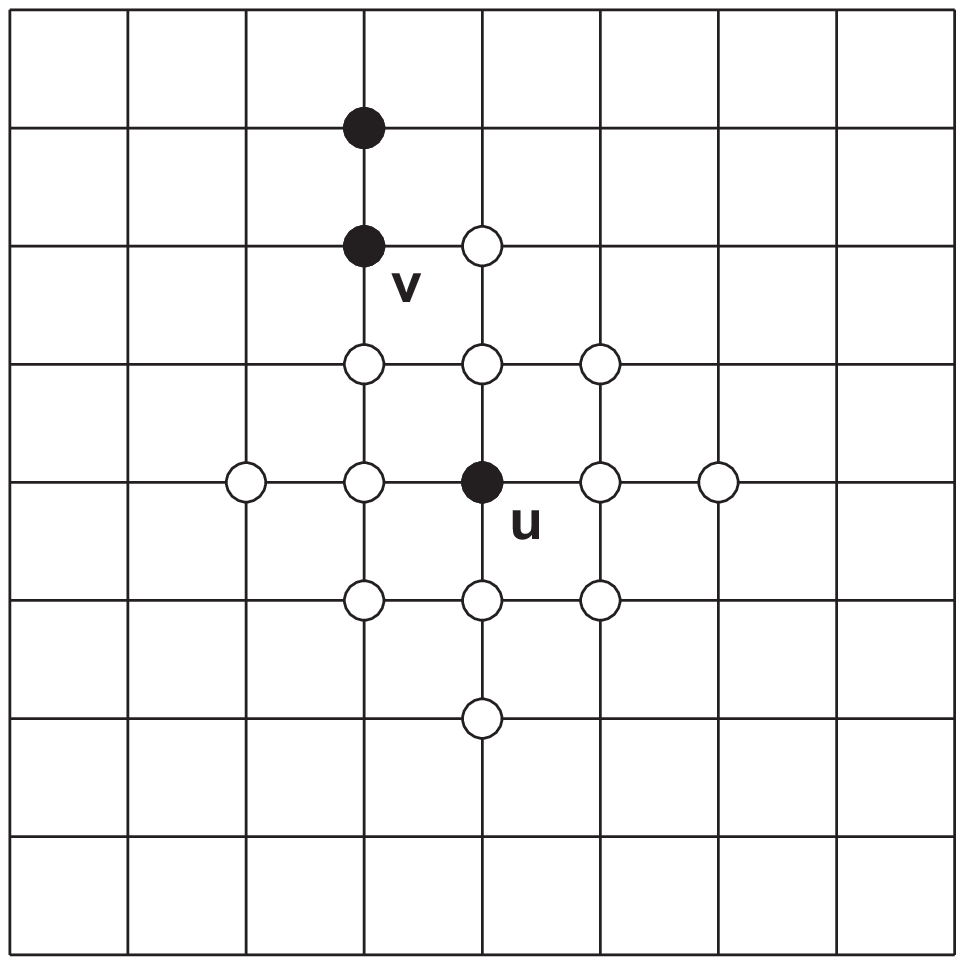}
\label{SquareR2Rule3}
}
\subfigure[Rule~4]{
\includegraphics[height=100pt]{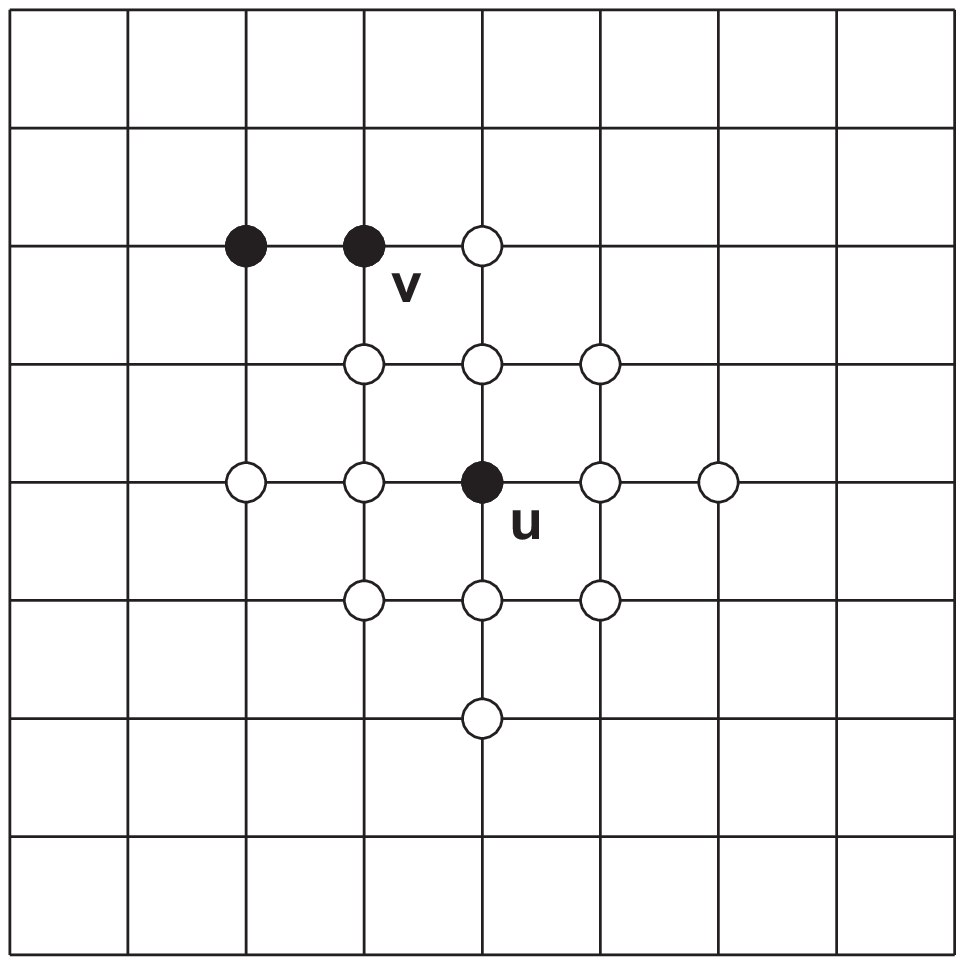}
\label{SquareR2Rule4}
}
\subfigure[Rule~5]{
\includegraphics[height=100pt]{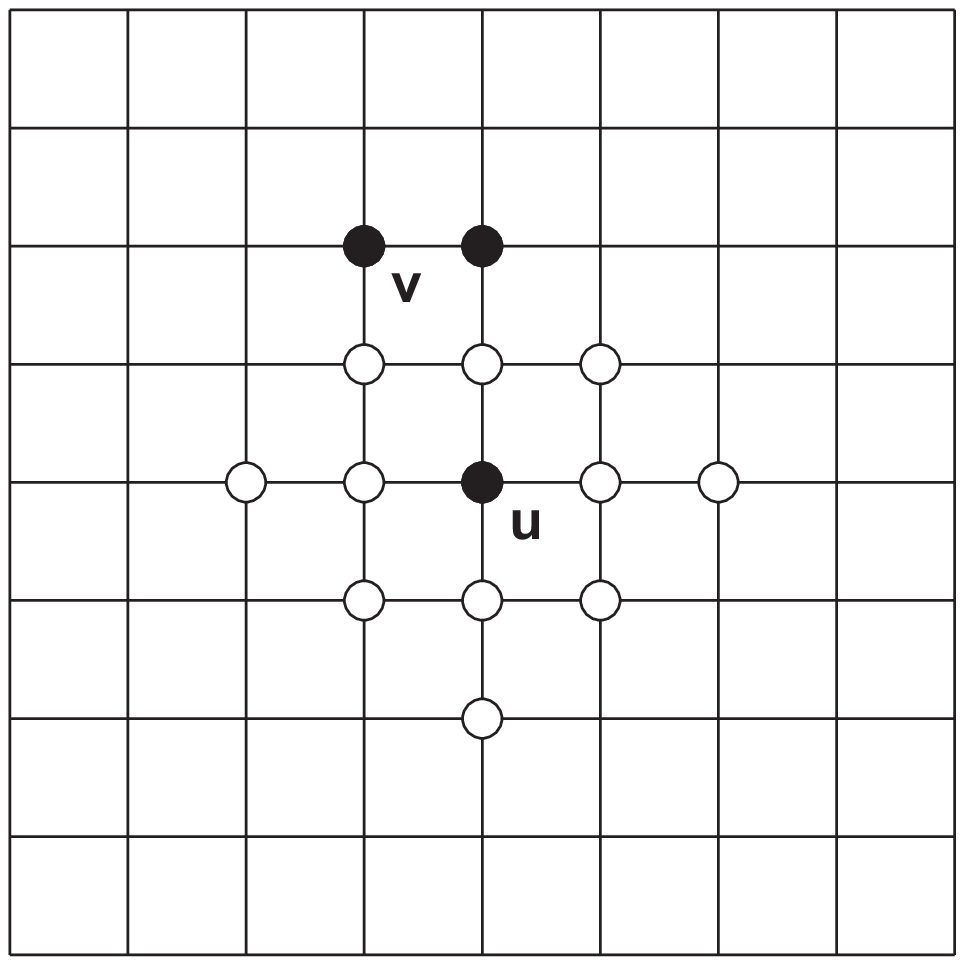}
\label{SquareR2Rule5}
}
\subfigure[Rule~6]{
\includegraphics[height=100pt]{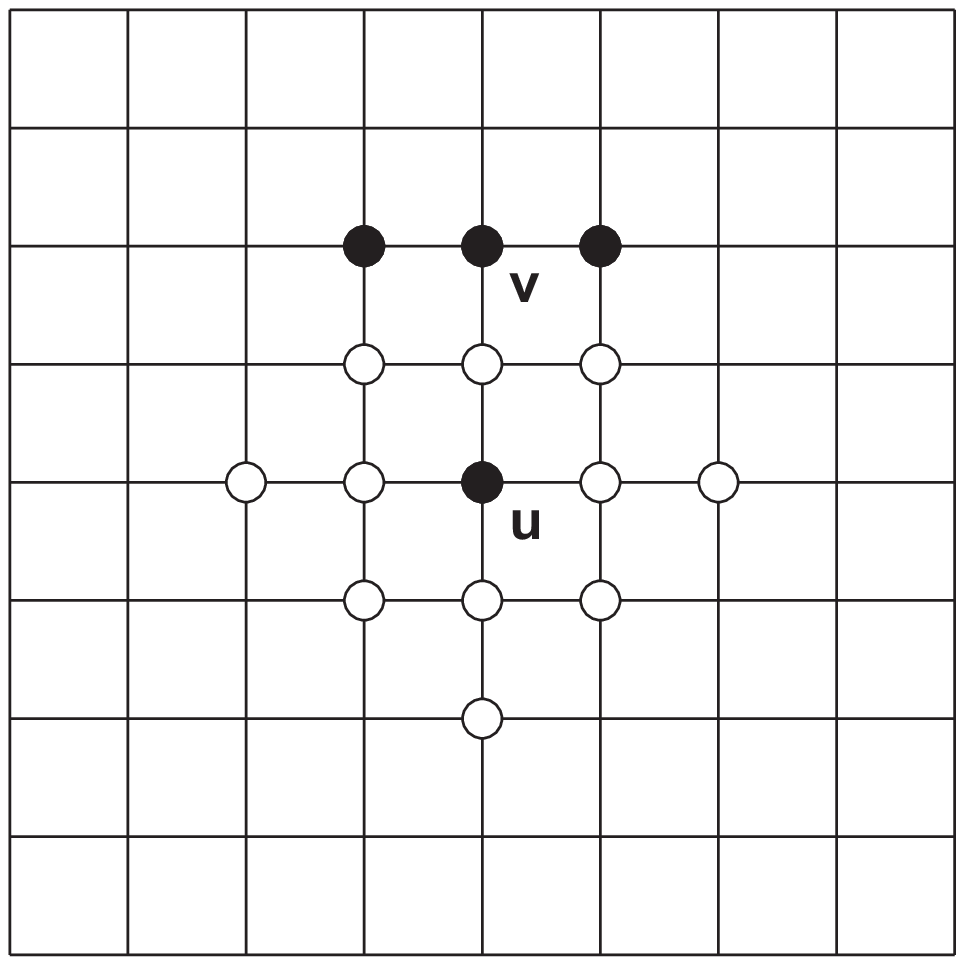}
\label{SquareR2Rule6}
}
\subfigure[Rule~7]{
\includegraphics[height=100pt]{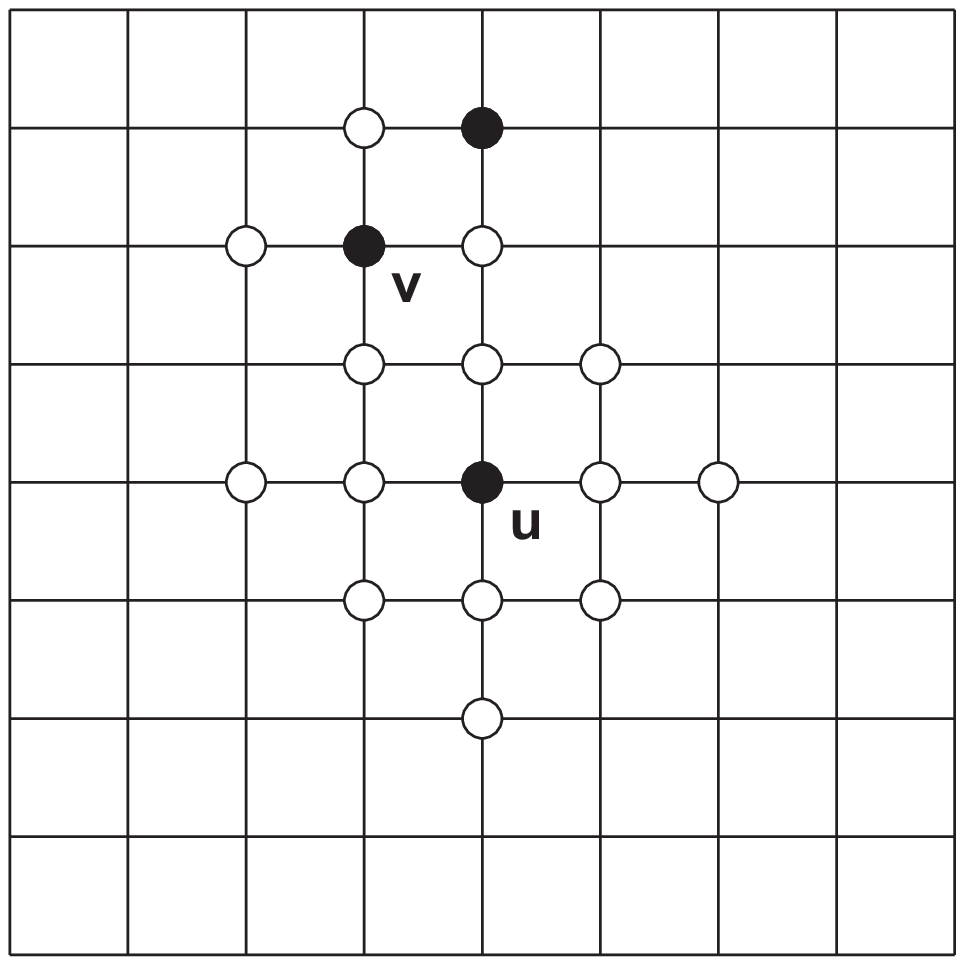}
\label{SquareR2Rule7}
}
\subfigure[Rule~8]{
\includegraphics[height=100pt]{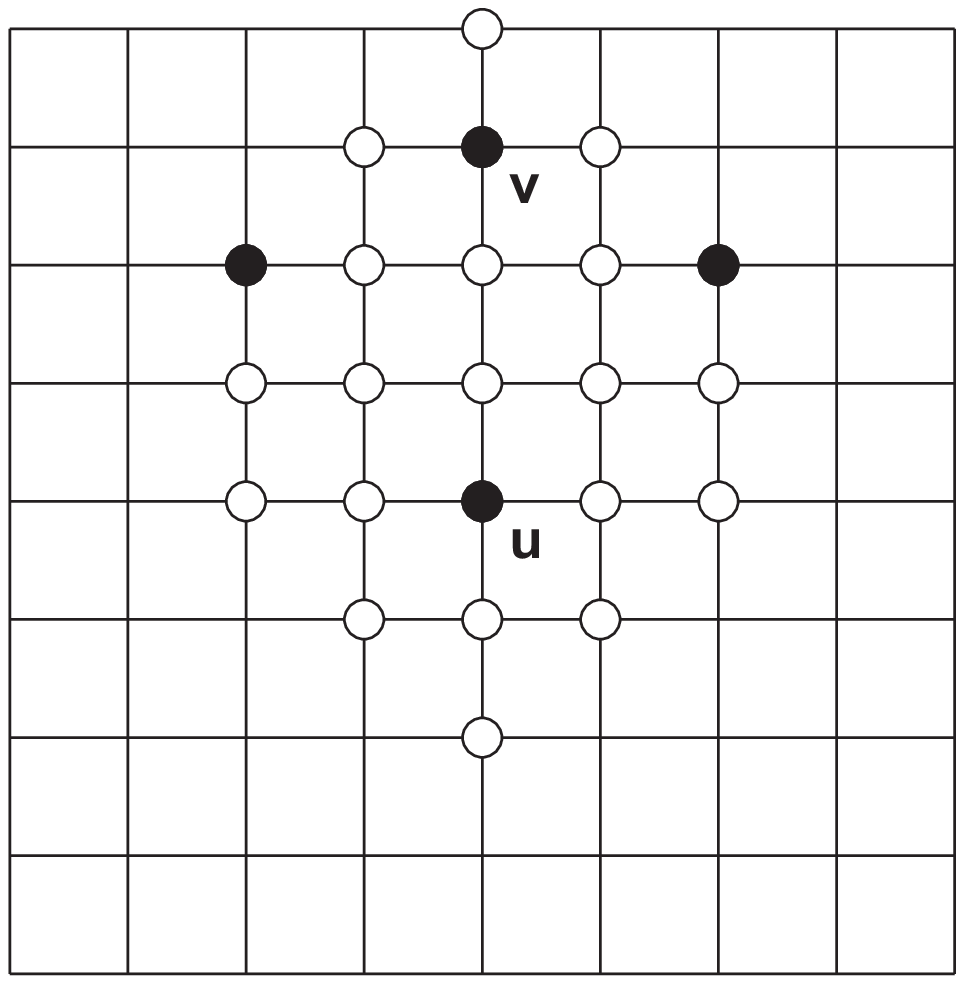}
\label{SquareR2Rule8}
}
\subfigure[Rule~9]{
\includegraphics[height=100pt]{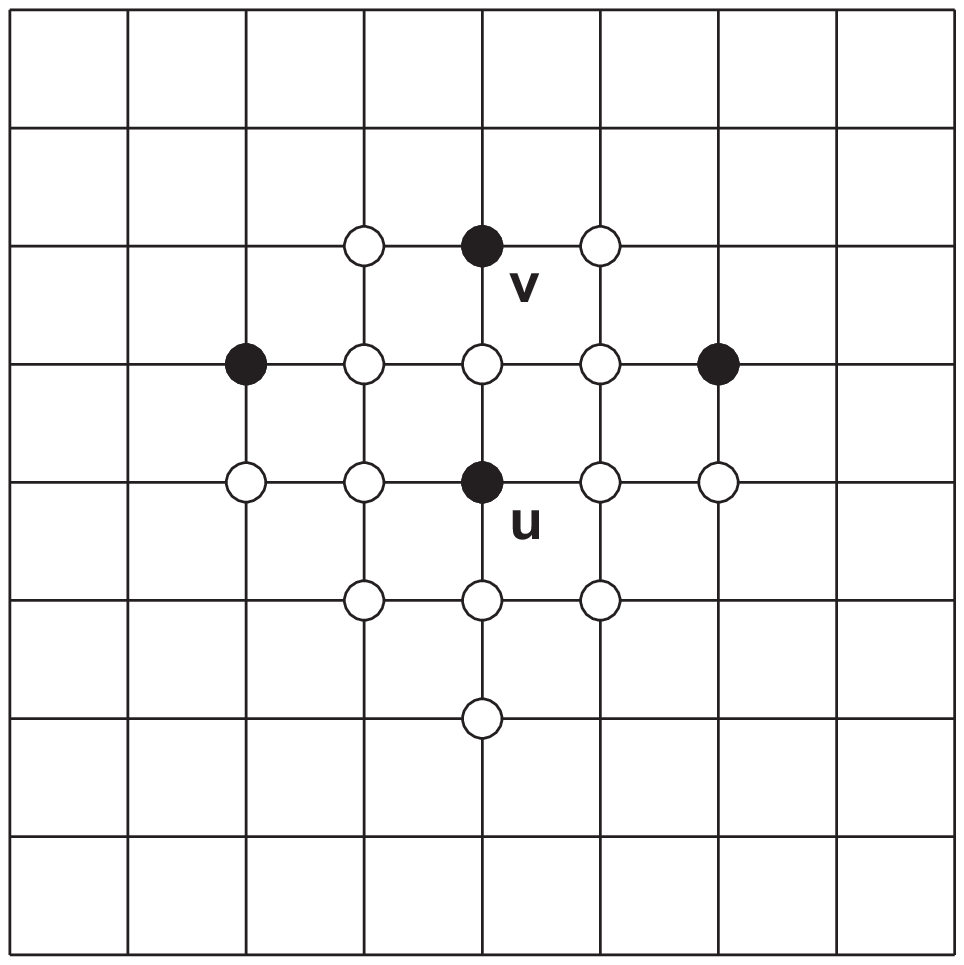}
\label{SquareR2Rule9}
}
\subfigure[Rule~10]{
\includegraphics[height=100pt]{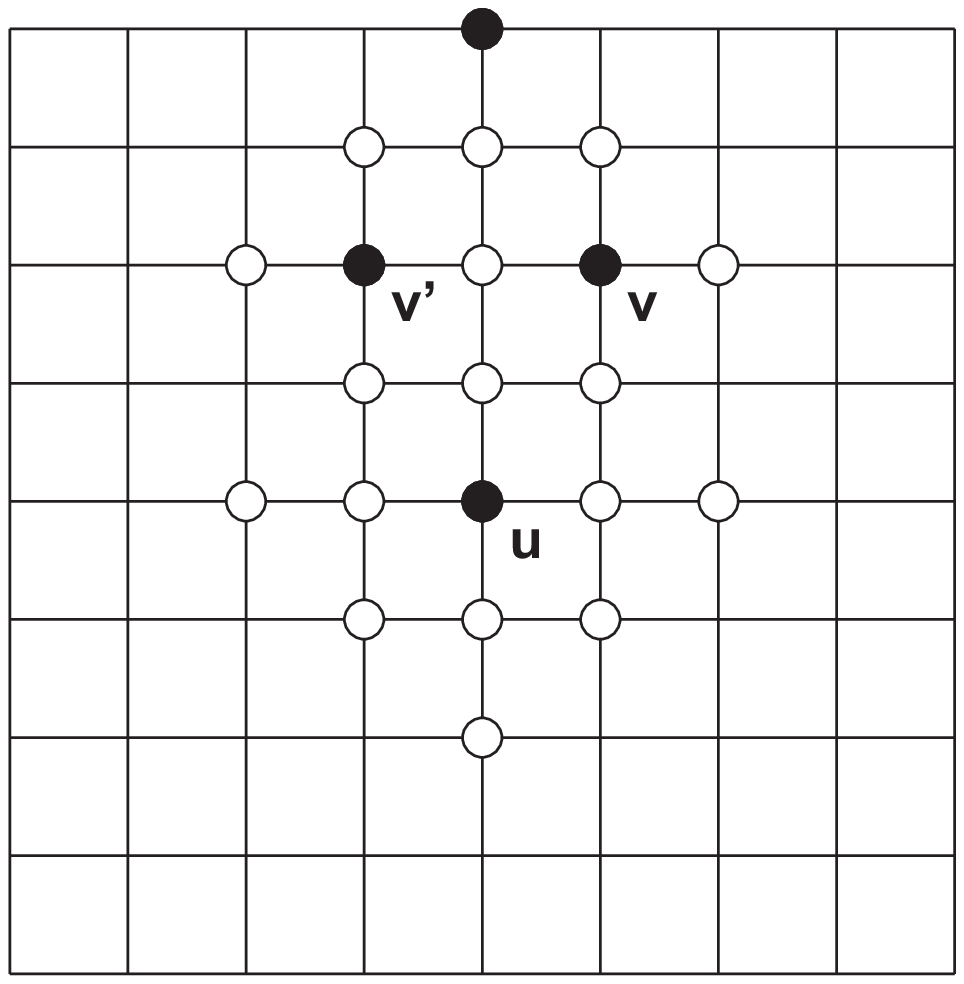}
\label{SquareR2Rule10}
}
\caption{The rules of the shifting scheme illustrated. The shaded dots represent codewords and the light dots represent non-codewords.}
\label{SquareR2RuleFigure}
\end{figure}

The \emph{rules} of the shifting scheme are illustrated in Figure~\ref{SquareR2RuleFigure}. In addition to the constellations in the figure, translations, rotations (by $\pi/2$, $\pi$ and $3\pi/2$ about the vertex $\bu$) and reflections (over the line passing vertically through $\bu$) can be applied to each rule in such a way that structure of the underlying graph $G_S$ is preserved. In the rules, share is shifted as follows:
\begin{itemize}
\item In Rule~1, if $I_2(\bv) \setminus I_2(\bv') \neq \emptyset$, then we shift $1/5$ units of share from $\bu$ to $\bv$, else $1/5$ units of share is shifted from $\bu$ to $\bv'$.
\item In Rules~2 and 5, we shift $1/30$ units of share from $\bu$ to $\bv$.
\item In Rule~3, we shift $1/12$ units of share from $\bu$ to $\bv$.
\item In Rules~4, 7 and 9, we shift $7/60$ units of share from $\bu$ to $\bv$.
\item In Rules~6 and 8, we shift $1/20$ units of share from $\bu$ to $\bv$.
\item In Rule~10, if $I_2(\bv) \setminus I_2(\bv') \neq \emptyset$, then we shift $7/60$ units of share from $\bu$ to $\bv$, else $7/60$ units of share is shifted to $\bv'$.
\end{itemize}

The modified share of a codeword $\bc \in C$, which is obtained after the shifting scheme is applied, is denoted by $\ms_2(\bc)$. The use of the rules is illustrated in the following simple example.
\begin{figure}
\centering
\subfigure[The first example]{
\includegraphics[height=150pt]{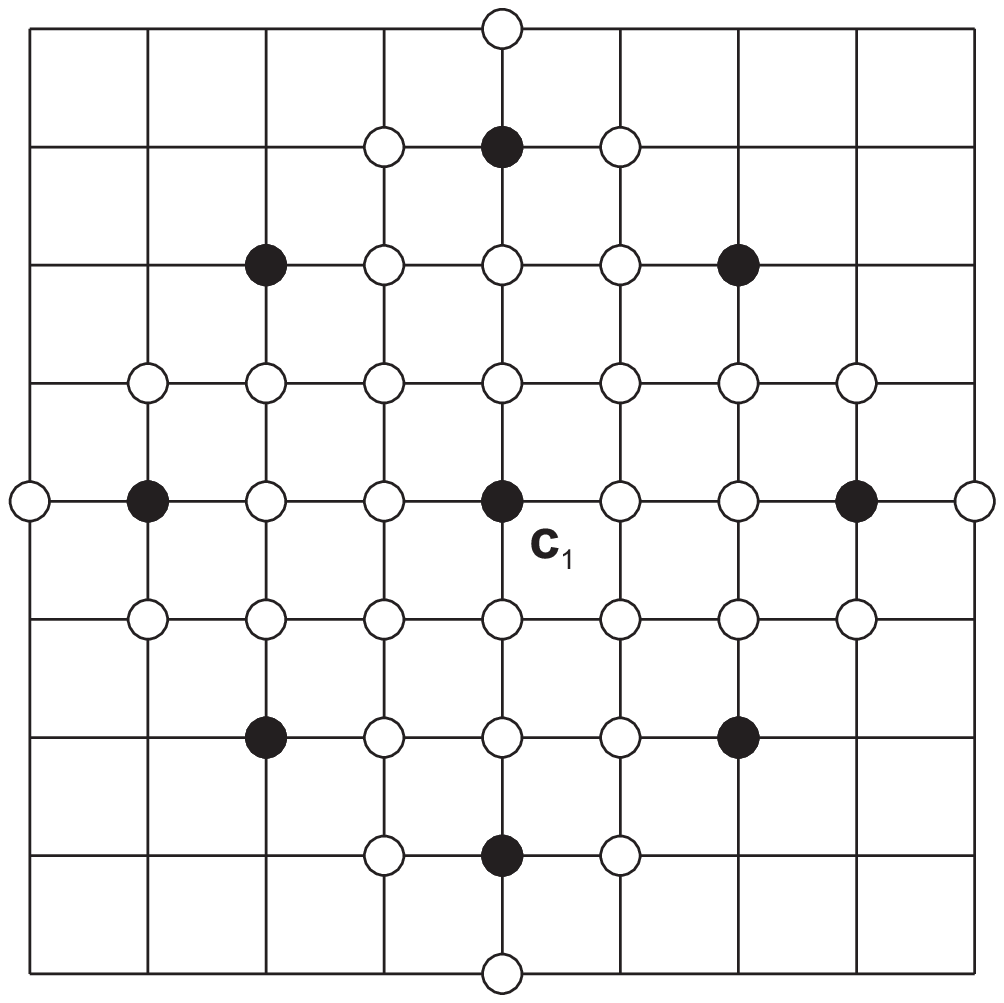}
\label{ExRulesUsage1}
}
\subfigure[The second example]{
\includegraphics[height=150pt]{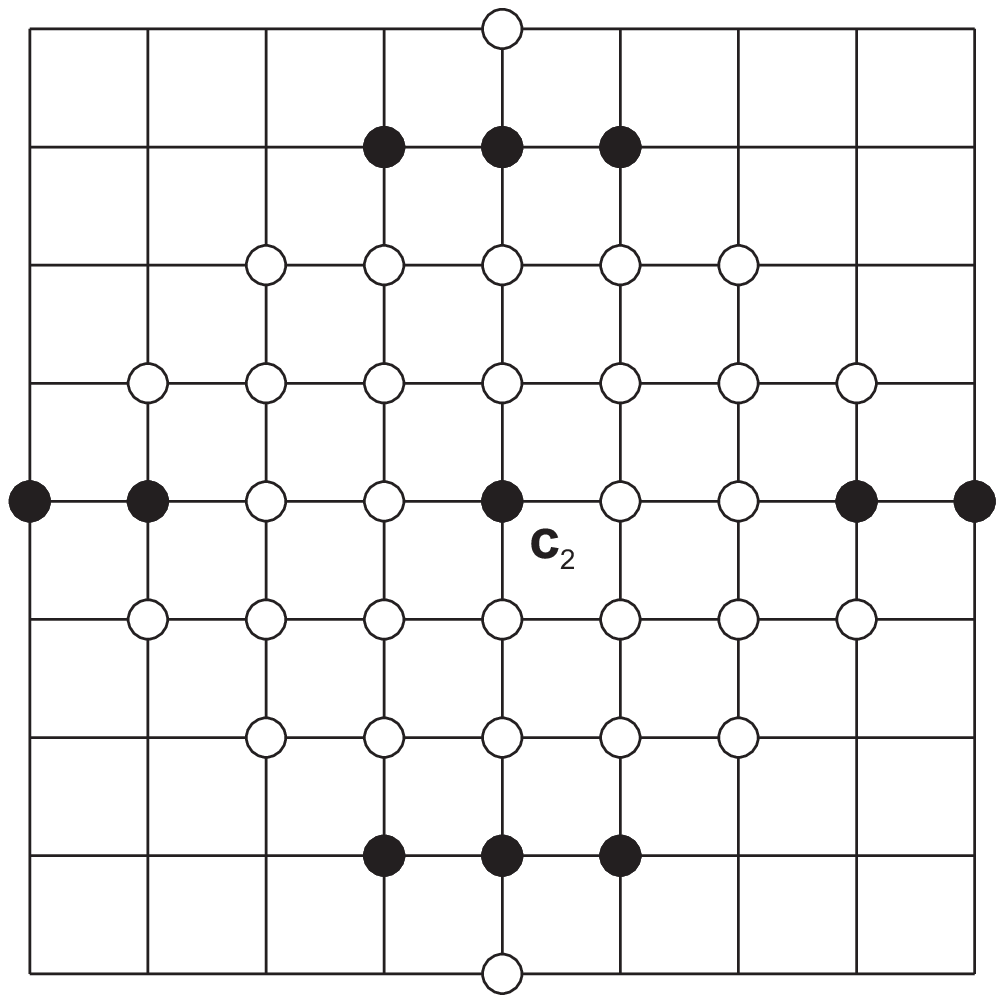}
\label{ExRulesUsage2}
}
\caption{The constellations of Example~\ref{ExampleForShifting} illustrated.}
\label{ShiftingExample}
\end{figure}
\begin{example} \label{ExampleForShifting}
Consider the codewords $\bc_1$ and $\bc_2$ with the constellations of codewords and non-codewords inside $B_4(\bc_1)$ and $B_4(\bc_2)$ as illustrated in Figure~\ref{ShiftingExample}. The shares of the codewords $\bc_1$ and $\bc_2$ are $6$ and $37/6$, respectively. Rule~8 can be applied to $\bc_1$ (four times). Hence, $1/20$ units of share is shifted from $\bc_1$ to each of the vertices $(-3,0)$, $(0,-3)$, $(0,3)$ and $(3,0)$. (Recall that translations, rotations and reflections can be applied to the constellations in Figure~\ref{SquareR2RuleFigure}.) Therefore, we have $\ms_2(\bc_1) = s_2(\bc_1) - 4 \cdot 1/20 = 6-1/5 = 29/5 \leq 35/6$.

In the case of the codeword $\bc_2$, share can be shifted from $\bc_2$ according to Rules~1 and 2. More precisely, share can be shifted twice from $\bc_2$ both to the vertices in $\{(-1,-3), (0,-3), (1,-3) \}$ and $\{(-1,3), (0,3), (1,3) \}$ according to Rule~1. Hence, $4 \cdot 1/5$ units of share is shifted according to Rule~1. Moreover, (in total) $2 \cdot 1/30$ units of share is shifted from $\bc_2$ to $(-3,0)$ and $(3,0)$ according to Rule~2. Therefore, we have $\ms_2(\bc_2) = s_2(\bc_2) - 4 \cdot 1/5 - 2 \cdot 1/30 = 37/6 - 4/5 -1/15 = 53/10 \leq 35/6$.
\end{example}

\subsection{The main theorem} \label{SquareSubsectionMainTheorem}

The following three lemmas show that $\ms_2(\bc) \leq 35/6$ for all $\bc \in C$. The proofs of the lemmas are postponed to Section~\ref{SquareSubsectionLemmasProofs}. It should be noted that the proofs of Lemmas~\ref{SquareR2ReceiveLemma2} and \ref{SquareR2ReceiveLemma1} are strictly mathematical proofs without any help from computers. However, in the proof of Lemma~\ref{SquareR2ReceivingLemma}, we also need to apply some exhaustive computer searches.
\begin{lemma} \label{SquareR2ReceiveLemma1}
Let $\bc \in C$ be a codeword such that $\bc$ is adjacent to another codeword and share is shifted to $\bc$ according to the rules. Then we have $\ms_2(\bc) \leq 35/6$.
\end{lemma}

\begin{lemma} \label{SquareR2ReceiveLemma2}
Let $\bc \in C$ be a codeword such that $\bc$ is not adjacent to another codeword and share is shifted to $\bc$ according to the rules. Then we have $\ms_2(\bc) \leq 35/6$.
\end{lemma}

\begin{lemma} \label{SquareR2ReceivingLemma}
Let $\bc \in C$ be a codeword such that no share is shifted to $\bc$
according to the rules. Then we have $\ms_2(\bc) \leq 35/6$.
\end{lemma}

\begin{thm} \label{SquareR2MainThm}
If $C$ is a $2$-identifying code in the square grid $G_S$, then we have
\[
D(C) \geq \frac{6}{35} \textrm{.}
\]
\end{thm}
\begin{proof}
Assume that $C$ is a $2$-identifying code in $G_S$. Recall that $Q_n = \{ (x,y) \in V_S \ | \ |x| \leq n, |y| \leq n \}$. Notice that each vertex $\bu \in Q_{n-2}$ with $|I_2(\bu)|=i$ contributes the summand $1/i$ to $s_2(\bc)$ for each of the $i$ codewords $\bc \in B_2(\bu)$. Therefore, we have
\begin{equation} \label{SquareR2MainEq1}
\sum_{\bc \in C \cap Q_n} s_2(\bc) \geq |Q_{n-2}| \textrm{.}
\end{equation}
Furthermore, we have
\begin{equation} \label{SquareR2MainEq2}
\sum_{\bc \in C \cap Q_n} s_2(\bc) \leq \sum_{\bc \in C \cap Q_n} \ms_2(\bc) + \frac{35}{6} |Q_{n+3} \setminus Q_n| \textrm{.}
\end{equation}
Indeed, shifting shares inside $Q_n$ does not affect the sum and each codeword in $Q_{n+3} \setminus Q_n$ can receive at most $35/6$ units of share (by Lemmas~\ref{SquareR2ReceiveLemma2} and \ref{SquareR2ReceiveLemma1}). Notice also that codewords in $Q_n$ cannot shift share to codewords outside $Q_{n+3}$. Therefore, combining the equations~\eqref{SquareR2MainEq1} and \eqref{SquareR2MainEq2} with the fact that $\ms_2(\bc) \leq 35/6$ for any $\bc \in C$, we obtain that
\[
\frac{|C \cap Q_n|}{|Q_n|} \geq \frac{6}{35} \cdot \frac{|Q_{n-2}|}{|Q_n|} - \frac{|Q_{n+3} \setminus Q_n|}{|Q_n|} \textrm{.}
\]
Since $|Q_k| = (2k+1)^2$ for any positive integer $k$, it is straightforward to conclude from the previous inequality that the density $D(C) \geq 6/35$.
\end{proof}

\subsection{The proofs of the lemmas} \label{SquareSubsectionLemmasProofs}

In what follows, we present the proofs of Lemmas~\ref{SquareR2ReceiveLemma1}, \ref{SquareR2ReceiveLemma2} and \ref{SquareR2ReceivingLemma}.

\begin{proof}[Proof of Lemma~\ref{SquareR2ReceiveLemma1}]
Let $\bc \in C$ be a codeword such that $\bc$ is adjacent to another codeword and share is shifted to $\bc$ according to the rules. Without loss of generality, we may assume that $\bc = (0,0)$. It is immediate that share can be shifted to $\bc$ only according to Rules~1--6 and 9. The proof now divides into four cases depending on the number of codewords adjacent to $\bc$.

\begin{figure}
\centering
\subfigure[Case (1)]{
\includegraphics[height=140pt]{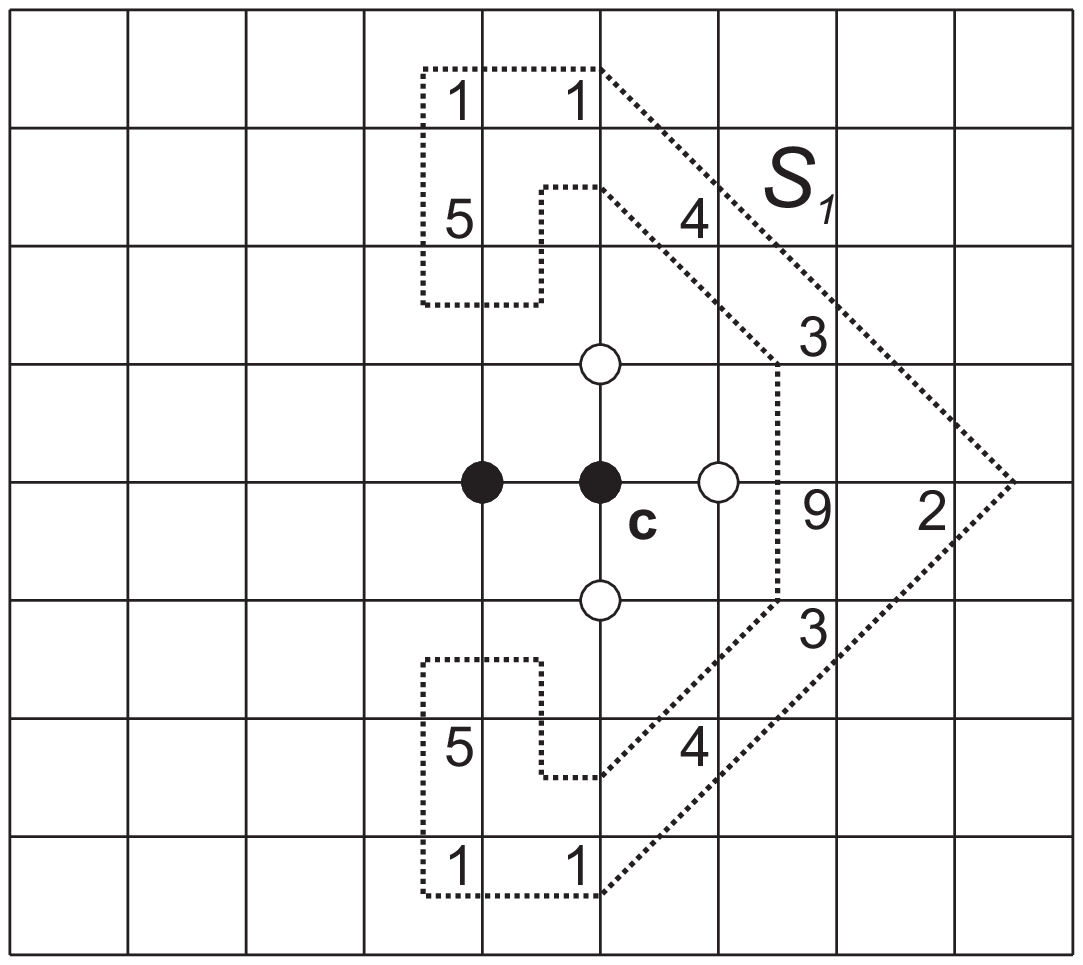}
\label{Lemma1Case1}
}
\subfigure[The first case of (2)]{
\includegraphics[height=140pt]{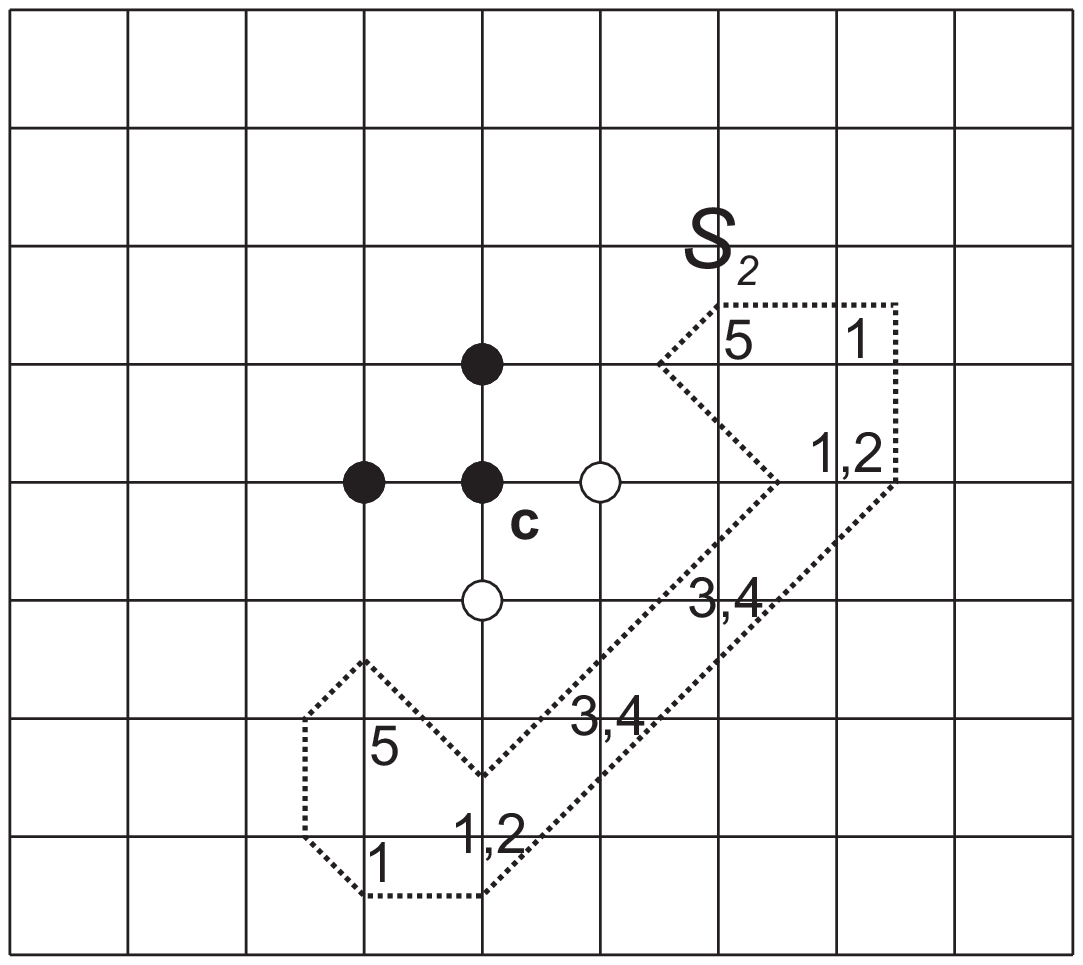}
\label{Lemma1Case2A}
}
\subfigure[The second case of (2)]{
\includegraphics[height=140pt]{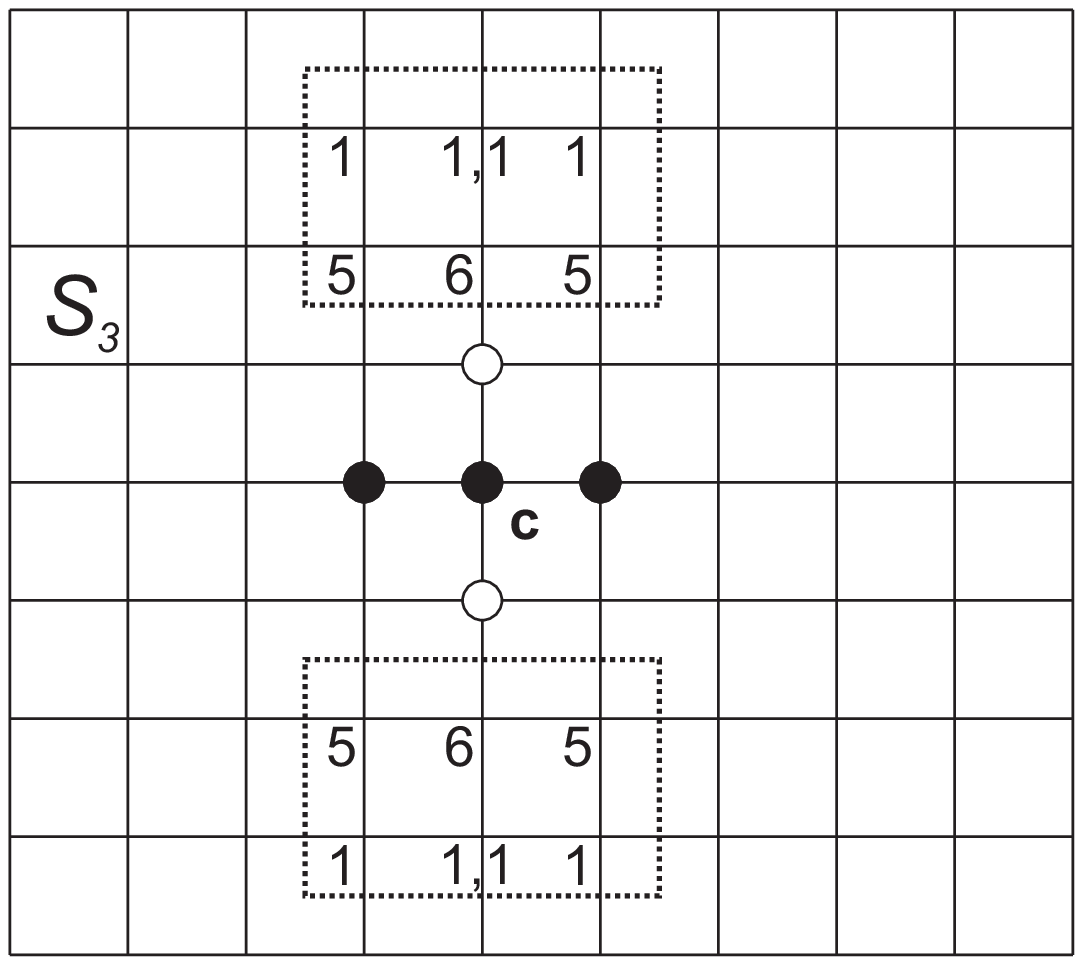}
\label{Lemma1Case2B}
}
\subfigure[Case (3)]{
\includegraphics[height=140pt]{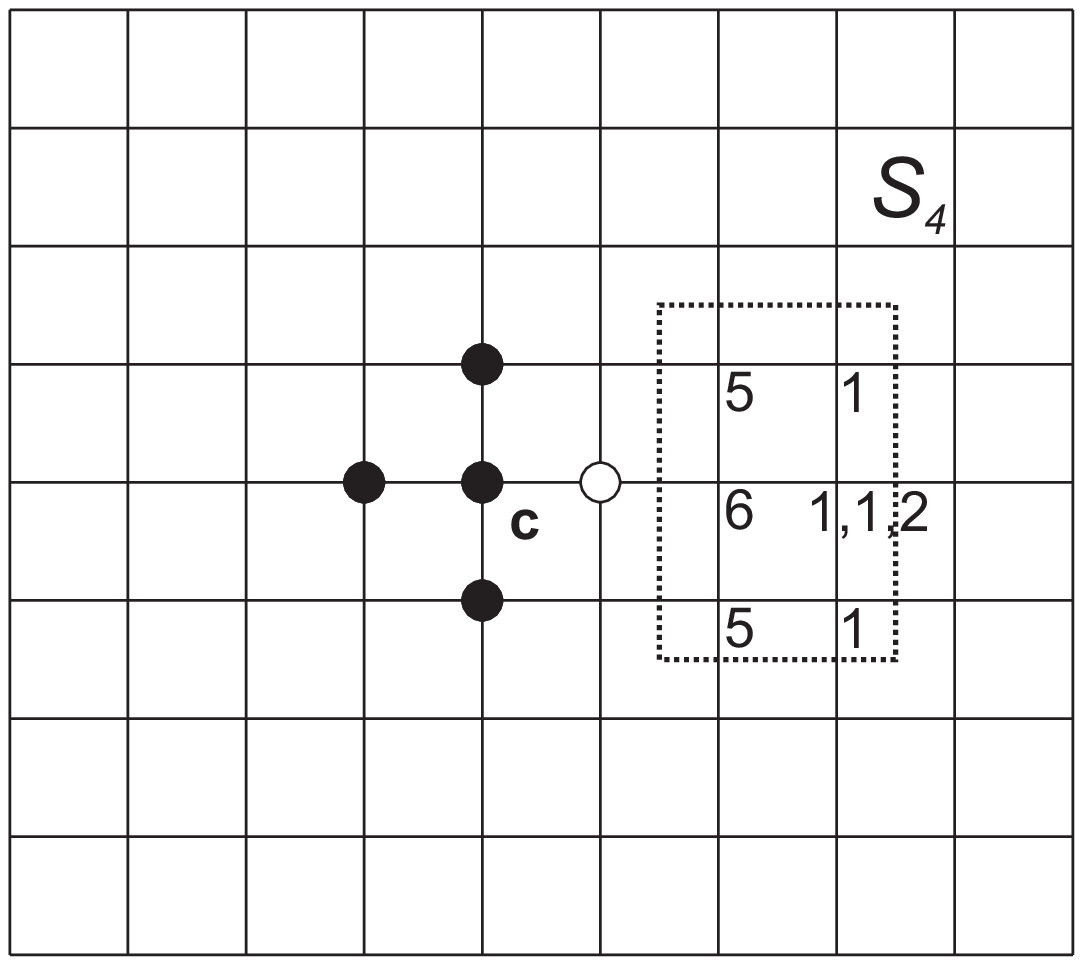}
\label{Lemma1Case3}
}
\caption{The cases of the proof of Lemma~\ref{SquareR2ReceiveLemma1} illustrated.}
\label{SquareSIllustrated}
\end{figure}

(1) Assume first that $\bc$ is adjacent to exactly one codeword. Now we may assume that the single codeword adjacent to $\bc$ is $(-1,0)$. Hence, the vertices $(1,0)$, $(0,-1)$ and $(0,1)$ are not codewords of $C$. Denote $X_1 = \{(-1,3), (-1,2), (0,3)\}$ and $X_2 = \{(-1,-3), (-1,-2), (0,-3)\}$. It is straightforward to conclude that share can be shifted to $\bc$ only from the vertices in the set
\[
S_1 = X_1 \cup X_2 \cup \{ (1, \pm 2), (2, \pm 1), (2,0), (3,0)\} \textrm{.}
\]
These observations are illustrated in Figure~\ref{Lemma1Case1}. In the figure, the number next to each vertex in $S_1$ state the rule according to which share is shifted. Notice also that if share is shifted from $X_i$, then exactly one of the vertices in $X_i$ is a codeword and other vertices are non-codewords. Therefore, at most $1/5$ units of share is shifted from each $X_i$. Since $C$ is a $2$-identifying code in $G_S$, the symmetric difference $B_2(-1,0) \, \triangle \, B_2(\bc)$ contains a codeword. 
Assuming first that at least one of the vertices in $B_2(\bc) \setminus B_2(-1,0)$ is a codeword, we obtain the following symmetrically different cases:
\begin{itemize}
\item Assume that $(2,0) \in C$. This implies that share cannot be shifted to $\bc$ from the vertices $(2,-1)$, $(2,1)$ and $(3,0)$. If $\bc$ receives share according to Rule~9, then the vertices $(1,-2)$ and $(1,2)$ belong to $C$. Now share can be shifted to $\bc$ only from the vertices $(-1,\pm 3)$, $(1,\pm 2)$ and $(2,0)$ according to Rules~1, 4 and 9, respectively. Choosing $C' = \{\bc, (-1,0),(1,-2),(1,2),(2,0) \}$ in Lemma~\ref{SimpleEstLemma}, we obtain that $s_2(\bc) \leq 73/15$. Therefore, we have $\ms_2(\bc) \leq s_2(\bc) + 2 \cdot 1/5 + 2 \cdot 7/60 + 7/60 \leq 337/60 \leq 35/6$.

    Assume then that $\bc$ does not receive share according to Rule~9. If both $(1,-2) \in C$ and $(1,2) \in C$, then as above we obtain that $\ms_2(\bc) \leq 35/6$. Assume then that exactly one of the vertices $(1,-2)$ and $(1,2)$ belongs to $C$. Without loss of generality, we may assume that $(1,2) \in C$ and $(1,-2) \notin C$. The share can be shifted to $\bc$ only from the vertex $(1,2)$ and from the sets $X_1$ and $X_2$. Therefore, at most $2 \cdot 1/5 + 7/60 = 31/60$ units of share is shifted to $\bc$. If share is shifted from the set $X_2$, then one of the vertices in $X_2$ is a codeword. Consequently, in each of the three cases, we have $s_2(\bc) \leq 59/12$ by Lemma~\ref{SimpleEstLemma}. Hence, we have $\ms_2(\bc) \leq s_2(\bc) + 31/60 \leq 163/30 \leq 35/6$. Thus, we may assume that $X_2 \cap C = \emptyset$ implying that at most $1/5 + 7/60 = 19/60$ units of share is shifted to $\bc$. By~Lemma~\ref{SimpleEstLemma} with the choice $C' = \{\bc, (-1,0),(2,0) \}$, we obtain that $s_2(\bc) \leq 65/12$. Therefore, we are done since $\ms_2(\bc) \leq s_2(\bc) + 19/60 \leq 86/15 \leq 35/6$. Thus, we may assume that $(1,2) \notin C$ and $(1,-2) \notin C$. Now at most $2/5$ units of share is shifted to $\bc$. We are again done since $\ms_2(\bc) \leq s_2(\bc) + 2/5 \leq 65/12 + 2/5 = 349/60 \leq 35/6$. 
\item Assume that $(2,0) \notin C$ and $(1,1) \in C$. By Example~\ref{SquareIllExampleAdjacentLemma}, we have $s_2(\bc) \leq 61/12$. The vertex $\bc$ can receive share from the vertices in $X_1 \cup X_2 \cup \{(1,-2), (2,-1), (3,0) \}$. It is easy to see that at most $\max\{ 1/12, 7/60 + 1/30 \}= 3/20$ units of share is shifted to $\bc$ from $\{(1,-2), (2,-1), (3,0) \}$. Therefore, $\bc$ receives in total at most $2 \cdot 1/5 + 3/20 =11/20$ units of share. Thus, $\ms_2(\bc) \leq s_2(\bc) + 11/20 \leq 61/12 + 11/20 = 169/30 \leq 35/6$ and we are done.
\item Assume that $(2,0) \notin C$, $(1,\pm 1) \notin C$ and $(0,2) \in C$. Now $\bc$ can receive share from the vertices in $X_2 \cup \{(1,-2), (2,\pm 1), (3,0) \}$. If $(1,-2) \in C$, then at most $1/5 + 7/60 + 1/12 = 2/5$ units of share is shifted to $\bc$. By Lemma~\ref{SimpleEstLemma} with the choice $C' = \{ \bc, (-1,0), (0,2), (1,-2)\}$, we obtain that $s_2(\bc) \leq 21/4$. Therefore, we are done since $\ms_2(\bc) \leq s_2(\bc) + 2/5 \leq 21/4 + 2/5 = 113/20 \leq 35/6$. If $(1,-2) \notin C$, then $\bc$ receives at most $1/5 + 1/12 = 17/60$ units share. Thus, since $s_2(\bc) \leq 11/2$ by Lemma~\ref{SimpleEstLemma}, we have $\ms_2(\bc) \leq s_2(\bc) + 17/60 \leq 11/2 + 17/60 = 347/60 \leq 35/6$ concluding the case.
\end{itemize}

Assume then that $B_2(\bc) \setminus B_2(-1,0)$ does not contain a codeword. Now share can be shifted to $\bc$ only from the vertices in $S_1 \setminus \{(-1, \pm 3), (0, \pm 3)\}$. Denote $Y_1 = \{(-1,2), (1,2), \linebreak (2,1)\}$ and $Y_2 = \{(-1,-2), (1,-2), (2,-1)\}$. It is easy to conclude that each set $Y_i$ can shift at most $\max\{ 1/12+1/30, 7/60\} = 7/60$ units of share to $\bc$. First let $(2,1)$ be a codeword of $C$. Choosing $C' = \{(\bc, (-1,0), (2,1)\}$ in Lemma~\ref{SimpleEstLemma}, we know that $s_2(\bc) \leq 67/12$. Moreover, $\bc$ receives at most $2 \cdot 7/60 = 7/30$ units of share (since share cannot be shifted from $(3,0)$). Thus, we are done since $\ms_2(\bc) \leq s_2(\bc) + 7/30 \leq 349/60 \leq 35/6$. Therefore, we may assume that $(2, -1) \notin C$ and $(2, 1) \notin C$. Now, if $Y_i$ shifts $7/60$ units of share to $\bc$, then the vertex $(1, (-1)^{i+1}\cdot 2)$ belongs to $C$. Consequently, if $7/60$ units of share is shifted from both $Y_1$ and $Y_2$, then $(1,\pm 2) \in C$ and by Lemma~\ref{SimpleEstLemma} we have $s_2(\bc) \leq 65/12$. Thus, we are done since $\ms_2(\bc) \leq s_2(\bc) + 2 \cdot 7/60 + 1/30 \leq 341/60 \leq 35/6$. Similarly, if $7/60$ units of share is shifted either from $Y_1$ or $Y_2$, then $s_2(\bc) \leq 67/12$ and we are done since $\ms_2(\bc) \leq s_2(\bc) + 7/60 + 2 \cdot 1/30 \leq 173/30 \leq 35/6$. Thus, we may assume that $(1,-2) \notin C$ and $(1,2) \notin C$. Furthermore, if $(-1,-2)$ or $(-1,2)$ belongs to $C$, then we have $s_2(\bc) \leq 17/3$ by Lemma~\ref{SimpleEstLemma}. Since now at most $3 \cdot 1/30 = 1/10$ units share can be shifted to $\bc$ from the vertices $(-1, \pm 2)$ and $(3,0)$ (other vertices are non-codewords), we obtain that $\ms_2(\bc) \leq s_2(\bc) + 1/10 \leq 173/30 \leq 35/60$. Thus, we may assume that $(-1, \pm 2) \notin C$ and, therefore, at most $1/30$ units of share can be shifted to $\bc$ from $(3,0)$. Recall that the set $B_2(-1,0) \setminus B_2(\bc)$ contains a codeword. With any choice for the codeword in this set, we obtain that $s_2(\bc) \leq 23/4$ by Lemma~\ref{SimpleEstLemma}. Thus, we have $\ms_2(\bc) \leq 23/4 + 1/30 = 347/60 \leq 35/6$. This completes the proof of the first case.

(2) Assume that $\bc$ is adjacent to exactly two codewords. The proof of the second case divides into the following symmetrically different cases:
\begin{itemize}
\item Assume first that $(-1,0), (0,1) \in C$ and $(1,0), (0,-1) \notin C$. We straightforwardly obtain that now share can be shifted to $\bc$ only from the vertices in
    \[
    S_2 = \{(-1,-3), (-1,-2), (0,-3), (1,-2), (2,-1), (2,1), (3,0), (3,1)\} \textrm{,}
    \]
    which is illustrated in Figure~\ref{Lemma1Case2A}. In the figure, the number(s) next to each vertex in $S_2$ state the rule(s) according to which share is shifted. As in Case~(1), we know that at most one vertex in each $X_2$ and $X_3 = \{(2,1),(3,0),(3,1)\}$ can shift share to $\bc$. Therefore, each $X_2$ and $X_3$ shifts at most $1/5 + 1/30 = 7/30$ units of share to $\bc$. Assume first that $(2,-1) \in C$. Then, by Lemma~\ref{SimpleEstLemma}, we have $s_2(\bc) \leq 151/30$. Moreover, at most $2 \cdot 7/30 + 1/12 + 7/60 = 2/3$ units of share is shifted to $\bc$. Therefore, we have $\ms_2(\bc) \leq s_2(\bc) + 2/3 \leq 57/10 \leq 35/6$ and we are done. Thus, by symmetry, we may assume that $(1,-2)$ and $(2,-1)$ do not belong to $C$. Then, by the previous observation, at most $2 \cdot 7/30$ units of share shifted to $\bc$. Choosing $C'= \{\bc, (-1,0), (0,1)\}$ in Lemma~\ref{SimpleEstLemma}, we obtain that $s_2(\bc) \leq 21/4$. Therefore, $\ms_2(\bc) \leq s_2(\bc) + 2 \cdot 7/30 \leq 343/60 \leq 35/6$ and we are done.
\item Assume then that $(-1,0),(1,0) \in C$ and $(0,-1),(0,1) \notin C$. Now share can be shifted to $\bc$ only from the vertices in
    \[
    S_3 =\{(x,y) \in V_S \ | \ x \in \{-1,0,1\}, y \in \{-3,-2,2,3\}\} \textrm{,}
    \]
    which is illustrated in Figure~\ref{Lemma1Case2B}. In the figure, the number(s) next to each vertex in $S_3$ state the rule(s) according to which share is shifted (if the same number is listed multiple times, then the corresponding rule can accordingly be used more than once). Using similar ideas as earlier, it can be concluded that at most $4 \cdot 1/5$ units of share is shifted to $\bc$. If any of the vertices $(-2,0)$, $(-1,-1)$, $(-1,1)$, $(1,-1)$, $(1,1)$ and $(2,0)$ is a codeword, then by Lemma~\ref{SimpleEstLemma} we have $s_2(\bc) < 5$ and we are immediately done. Hence, we may assume that none of these vertices is a codeword. If $(0,2)$ belongs to $C$, then at most $1/20$ units of share can be shifted from the upper part of $S_3$. Therefore, at most $2 \cdot 1/5 + 1/20$ units of share is shifted to $\bc$ and we are done since $s_2(\bc) \leq 307/60$. Hence, we may assume that $(0,-2) \notin C$ and $(0,2) \notin C$. Consequently, we know that the sets $B_2(\bc) \setminus B_2(-1,0)$ and $B_2(\bc) \setminus B_2(1,0)$ do not contain any codewords. This implies that share cannot be shifted to $\bc$ according to Rule~1. Therefore, share can be shifted from at most two vertices of $S_3$. Hence, at most $2 \cdot 1/20$ units of share is shifted to $\bc$. Thus, we are done since $s_2(\bc) \leq 31/6$ (by Lemma~\ref{SimpleEstLemma} with the choice $C' = \{\bc, (-1,0), (1,0)\}$).
\end{itemize}

(3) Assume that $\bc$ is adjacent to exactly three codewords. Without loss of generality, we may assume that $(-1,0),(0,-1),(0,1) \in C$ and $(1,0) \notin C$. Now share can be shifted to $\bc$ only from the vertices in
\[
S_4 =\{(x,y) \in V_S \ | \ x \in \{2,3\}, y \in \{-1,0,1\}\} \textrm{,}
\]
which is illustrated in Figure~\ref{Lemma1Case3}. In the figure, the number(s) next to each vertex in $S_4$ state the rule(s) according to which share is shifted (if the same number is listed multiple times, then the corresponding rule can accordingly be used more than once). By Lemma~\ref{SimpleEstLemma} with the choice $C' = \{\bc, (-1,0), (0,-1), (0,1)\}$, we obtain that $s_2(\bc) \leq 293/60$. It is immediate that the maximum amount of share is shifted to $\bc$ if the vertex $(3,0)$ is a codeword. Then at most $2 \cdot 1/5 + 1/30 = 13/30$ units of share is shifted. Thus, we have $\ms_2(\bc) \leq 293/60 + 13/30 \leq 35/6$.

(4) Finally, assume that $\bc$ is adjacent to exactly four codewords, i.e. all the adjacent vertices are codewords. Now share cannot be shifted to $\bc$ according to any rule. Therefore, we are done since $\ms_2(\bc) \leq s_2(\bc) \leq 21/5 \leq 35/6$. This concludes the proof of the lemma.
\end{proof}

\begin{proof}[Proof of Lemma~\ref{SquareR2ReceiveLemma2}]
Let $\bc \in C$ be a codeword such that $\bc$ is not adjacent to another codeword and share is shifted to $\bc$ according to the rules. Without loss of generality, we may assume that $\bc = (0,0)$. It is immediate that share can be shifted to $\bc$ only according to Rules~7--10.

Assume first that $\bc$ receives share according to Rule~10. Without of loss of generality, we may assume that the rule is applied as in Figure~\ref{SquareR2Rule10} with $\bc = \bv$. Recall that now there exists a codeword in $B_2(\bc) \setminus B_2(\bv')$. Let us then show that in this case Rules~7 and 8 cannot be applied to $\bc$. Assume to the contrary that Rule~7 is applied. Then the vertex $(1,-1)$ or $(1,1)$ belongs to $C$. If $(1,-1) \in C$, then share is shifted to $\bc$ from the vertices $(-2,-1)$ and $(1,2)$. However, this is impossible since $(-2,-1) \notin C$ and $(-1,2) \in C$, respectively. If $(1,1) \in C$, then, using reasoning analogous to above, we obtain a contradiction. Thus, Rule~7 cannot be applied to $\bc$. If Rule~8 is applied to $\bc$, then share is shifted to $\bc$ from the vertices $(-3,0)$, $(0,-3)$, $(0,3)$ or $(3,0)$. All these cases lead to contradictions since $(-3,0) \notin C$, $(0,-3) \notin C$, $(-1,2) \in C$ and $(1,-2) \notin C$, respectively. Using similar argumentation, it can be shown that Rule~9 can be applied to $\bc$ only in such a way that share is shifted to $\bc$ from the vertex $\bv' = (-2,0)$. Moreover, it is clear that share can be shifted to $\bc$ according to Rule~10 at most twice; namely, from the vertices $(-1,-2)$ and $(-1,2)$. Thus, at most $3 \cdot 7/60$ units of share can be shifted to $\bc$. Since $B_2(\bc) \setminus B_2(\bv')$ contains a codeword, at least one of the vertices $(1,-1)$, $(2,0)$, $(1,1)$ and $(0,2)$ is a codeword. If $(1,-1)$, $(2,0)$ or $(1,1)$ belongs to $C$, then $s_2(\bc) \leq 59/12$ and we are done since $\ms_2(\bc) \leq 59/12 + 3 \cdot 7/60 = 79/5 \leq 35/6$. If $(0,2) \in C$, then share cannot be shifted to $\bc$ from $(-1,2)$ according to Rule~10 and $s_2(\bc) \leq 67/12$. Hence, $\ms_2(\bc) \leq 67/12 + 2 \cdot 7/60 = 349/60 \leq 35/6$ and we are done.

Assume that $\bc$ receives share according to Rule~9 and that the rule is applied as in Figure~\ref{SquareR2Rule9} with $\bc = \bv$. Let us first show that now Rule~7 cannot be applied to $\bc$. Assuming to the contrary that Rule~7 can be applied, we know that $(-1,1)$ or $(1,1)$ belongs to $C$. By symmetry, we may assume that $(1,1) \in C$. Then share is received from the vertices $(-2,1)$ and $(1,-2)$. However, the facts that $(-2,-1) \in C$ and $(1,-2) \notin C$ respectively lead to contradictions. Hence, Rule~7 cannot be applied. If Rule~10 is applied to $\bc$, then the reasoning reduces to the one in the previous paragraph. If Rule~8 is applied to $\bc$, then $(-2,1)$, $(0,3)$ and $(2,1)$ belong to $C$ and share is shifted to $\bc$ from $(0,3)$. By Lemma~\ref{SimpleEstLemma}, we have $s_2(\bc) \leq 14/3$. Furthermore, it is easy to verify that Rules~8 and 9 can be applied to $\bc$ at most once. Therefore, we are done since $\ms_2(\bc) \leq s_2(\bc) + 7/60 + 1/20 \leq 35/6$. Thus, we may assume that only Rule~9 is applied to $\bc$. The rule can clearly be applied at most twice. Moreover, if the rule is used twice, then $(-2,1)$, $(0,2)$ and $(2,1)$ belong to $\bc$ and we are immediately done as above. Finally, assume that the rule is used only once. Since the vertices $\bc$ and $(0,-2)$ are $2$-separated, one of the vertices $(\pm 2, 0)$, $(\pm 1, 1)$, $(0,1)$ and $(0,2)$ is a codeword. In each case, $s_2(\bc) \leq 67/12$ and we are done since $\ms_2(\bc) \leq s_2(\bc) + 7/60 \leq 35/6$.

Assume that $\bc$ receives share according to Rule~8 and that the rule is applied as in Figure~\ref{SquareR2Rule8} with $\bc = \bv$. By the previous cases, we know that Rule 10 cannot be used. Similarly to the previous cases, we can also show that Rule~7 cannot be applied to $\bc$. Moreover, by the previous paragraph, we may assume that Rule~9 is not applied to $\bc$. If share is shifted to $\bc$ according to Rule~8 more than once, then the vertices $(-2,1)$, $(0,3)$ and $(2,1)$ are codewords of $C$ and we have $s_2(\bc) \leq 17/3$. Therefore, since at most $2 \cdot 1/20 = 1/10$ units of share is shifted to $\bc$, we have $\ms_2(\bc) \leq 17/3 + 1/10 = 173/30 \leq 35/6$. Thus, we may assume that only Rule~8 is applied to $\bc$ and only once. Since at most $1/20$ units of share is shifted to $\bc$, it is enough to show that $s_2(\bc) \leq 23/4 \leq 35/6 - 1/20$. Indeed this is the case, if $(\pm 2, 0)$, $(\pm 1, 1)$ or $(\pm 1, 2)$ belongs to $C$. Hence, we may assume that these vertices are not codewords of $C$. Since $C$ is a $2$-identifying code in $G_S$, the symmetric differences $B_2(-1,-1) \, \triangle \, B_2(-1,0)$, $B_2(0,0) \, \triangle \, B_2(0,1)$ and $B_2(1,-1) \, \triangle \, B_2(1,0)$ each contain a codeword. In other words, this means that each set $X_1 = \{(-3,-1), (-3,0), (-2,1)\}$, $X_2 = \{(-2,1), (0,3), (2,1)\}$ and $X_3 = \{(3,-1), (3,0), (2,1)\}$ contains a codeword, respectively. If $(-2,1) \in C$, then also one of the vertices in $X_3$ is a codeword. Consequently, in each case, we have $s_2(\bc) \leq 23/4$. Thus, we may assume that $(-2,1) \notin C$ and $(2,1) \notin C$. Then $(0,3)$ belongs to $C$ since $I_2(0,0) \, \triangle \, I_2(0,1) \neq \emptyset$. Furthermore, there exists a codeword both in set $X_1 \setminus \{(-2,1)\}$ and in set $X_3 \setminus \{(2,1)\}$ . In all the remaining cases, we have $s_2(\bc) \leq 23/4$.

Assume that $\bc$ receives share according to Rule~7 and that the rule is applied as in Figure~\ref{SquareR2Rule7} with $\bc = \bv$. By the previous cases, we know that Rules~8, 9 and 10 cannot be applied to $\bc$. Thus, $\bc$ can receive share only according to Rule~7. Moreover, using similar reasoning as above, it can be shown that share can be shifted to $\bc$ according to Rule~7 only from $(1,-2)$, and either from $(-2,1)$ or $(-1,2)$. Hence, at most $2 \cdot 7/60 = 7/30$ units of share is shifted to $\bc$. By Lemma~\ref{SimpleEstLemma}~with the choice $C'= \{\bc, (1,-2), (1,1)\}$, we have $s_2(\bc) \leq 67/12$. Therefore, we are done since $\ms_2(\bc) \leq 67/12 + 7/30 \leq 349/60 \leq 35/6$. This completes the proof of the lemma.
\end{proof}

\begin{proof}[Proof of Lemma~\ref{SquareR2ReceivingLemma}]
Let $\bc \in C$ be a codeword such that no share is shifted to $\bc$
according to the rules. Assume first that $\bc$ is adjacent to another codeword. Without loss of generality, we may assume that $\bc = (0,0)$ and $(-1,0) \in C$. Now, choosing $C'= \{ \bc, (-1,0) \}$ in Lemma~\ref{SimpleEstLemma}, we obtain that $s_2(\bc) \leq 35/6$. Hence, we have $\ms_2(\bc) \leq s_2(\bc) \leq 35/6$. Assume then that at least one of the vertices at Euclidean distance $\sqrt{2}$ from $\bc$ is a codeword. Without loss of generality, we may assume that $(-1,1) \in C$. Again, by Lemma~\ref{SimpleEstLemma}, we obtain that $s_2(\bc) \leq 35/6$ and, therefore, we are done.

Thus, we may assume that the vertices at Euclidean distance $1$ and $\sqrt{2}$ from $\bc$ are non-codewords. Then denote $A = \{(-2,0), (0,-2), (0,2), (2,0)\}$. If at least three of the vertices in $A$ are codewords, then by Lemma~\ref{SimpleEstLemma} $s_2(\bc) \leq 67/12 \leq 35/6$ and we are done. Assuming there are exactly two codewords in $A$, we obtain the following symmetrically different cases:
\begin{itemize}
\item Assume that $(-2,0)$ and $(2,0)$ belongs to $C$. Then $(0,-2)$ and $(0,2)$ do not belong to $C$. Since $C$ is a $2$-identifying code, the symmetric difference $B_2(0,-1) \, \triangle \, B_2(0,1)$ contains a codeword. Without loss of generality, we may assume that at least one of the vertices $(-2,1)$, $(-1,2)$ and $(0,3)$ is a codeword. In each case, we have $s_2(\bc) \leq 17/3 \leq 35/6$. Hence, we are done.
\item Assume that $(0,2)$ and $(2,0)$ belongs to $C$. The symmetric difference $B_2(-1,0) \, \triangle \linebreak B_2(0,-1)$ contains a codeword. Without loss of generality, we may assume that at least one of the vertices $(0,-3)$, $(1,-2)$ and $(2,-1)$ is a codeword. In each case, we again have $s_2(\bc) \leq 23/4 \leq 35/6$ and we are done.
\end{itemize}

We may now assume that at most one vertex of the set $A$ is a codeword. Namely, we have the following symmetrically different cases: $I_2(\bc) = \{\bc\}$ and $I_2(\bc) = \{\bc, (2,0)\}$. Notice first that since $C$ is a $2$-identifying code in $G_S$, the pairs of vertices $\bu$ and $\bv$ belonging to $B_2(\bc)$ are $2$-separated by the codewords contained in $B_4(\bc)$. Therefore, in order to prove that always $\ms_2(\bc) \leq 35/6$, it is enough to check that for each subset $F \subseteq B_4(\bc) \setminus B_2(\bc)$ such that the pairs of vertices $\bu, \bv \in B_2(\bc)$ are $2$-separated by $F \cup I_2(\bu)$, we have $\ms_2(\bc) \leq 35/6$ after the rules have been applied. Notice that the rules according to which share is shifted from $\bc$ depend only on the constellations of codewords and non-codewords inside $B_4(\bc)$. For both cases $I_2(\bc) = \{\bc\}$ and $I_2(\bc) = \{\bc, (2,0)\}$, we need to go through all the subsets of $B_4(\bc) \setminus B_2(\bc)$. Hence, for both cases, we have $2^{28}$ different sets to consider (since $|B_4(\bc) \setminus B_2(\bc)|=28$).

Depending on the implementation of the exhaustive computer search, the explained brute-force method might be a little bit too inefficient for practical purposes. In what follows, we explain a more sophisticated approach to execute the computations. The actual program code of this method is presented in Appendix. Notice first that the pairs of vertices $\bu$ and $\bv$ belonging to $B_1(\bc)$ are $2$-separated by the codewords contained in $B_3(\bc)$. For each set $D' \subseteq B_3(\bc) \setminus B_2(\bc)$ such that $D' \cup I_2(\bc)$ $2$-separates all the pairs in $B_1(\bc)$, we calculate the upper bound given by Lemma~\ref{SimpleEstLemma} when $D = D' \cup I_2(\bc)$ and if this upper bound is greater than $35/6$, then we add $D$ to the collection of sets $\St$. For the identifying sets $I_2(\bc) = \{\bc\}$ and $I_2(\bc) = \{\bc, (2,0)\}$, there exists 209 and 35 sets in the collection $\St$, respectively. For each set $D \in \St$, we then check that for all the subsets $F' \subseteq B_4(\bc) \setminus B_3(\bc)$ such that $F' \cup D$ $2$-separates all the pairs in $B_2(\bc)$, we have $\ms_2(\bc) \leq 35/6$ after the rules have been applied. This approach reduces the number of cases to $209 \cdot 2^{16}$ and $35 \cdot 2^{16}$ (since $|B_4(\bc) \setminus B_3(\bc)|=16$) when $I_2(\bc) = \{\bc\}$ and $I_2(\bc) = \{\bc, (2,0)\}$, respectively. This exhaustive computer search concludes the proof of the lemma.
\end{proof}

\section*{Appendix}

The computations needed in the proof of Lemma~\ref{SquareR2ReceivingLemma} have been executed with \emph{Mathematica}. The Mathematica notebook of the computations is available in \cite{JunnilaMathematica}. In what follows, the program code is presented with some guiding comments. The following basic functions are used later in the code:
\begin{itemize}
\item Function outputting a set $\{ (x,y) \in \Z^2 \ \, | \ \, |x|, |y| \leq n \}$.
\begin{verbatim}
CreateGrid[n_] :=
  Module[{G = {}, i, j},
   For[i = -n, i < n + 1, i++,
    	For[j = -n, j < n + 1, j++,
      		G = Union[G, {{i, j}}];
      ];
    ];
   Return[G];
   ];
\end{verbatim}
\item Function outputting $B_r((x,y))$ in a square grid.
\begin{verbatim}
rBallSquare[r_, x_, y_] :=
  Module[{B = {}, i, j},
   For[i = x - r, i < x + r + 1, i++,
    	For[j = y - r, j < y + r + 1, j++,
      		If[Abs[i - x] + Abs[j - y] <= r, B = Union[B, {{i, j}}]];
      ];
    ];
   Return[B];
   ];
\end{verbatim}\newpage
\item Function outputting $B_r((x,y))$ in a king grid, i.e., in a square grid with diagonals.
\begin{verbatim}
rBallKing[r_, x_, y_] :=
  Module[{B = {}, i, j},
   For[i = x - r, i < x + r + 1, i++,
    	For[j = y - r, j < y + r + 1, j++,
      		B = Union[B, {{i, j}}];
      ];
    ];
   Return[B];
   ];
\end{verbatim}
\item Function testing whether for each vertex $\bu \in J$ the intersection of $B_r(\bu)$ and the code $K$ is nonempty and unique.
\begin{verbatim}
IDonSquareGrid[K_, J_, r_] :=
  Module[{NoID = False, S = {}, i, L = {}},
   For[i = 1, i < Length[J] + 1, i++,
    L = Intersection[rBallSquare[r, J[[i]][[1]], J[[i]][[2]]], K];
    If[Length[L] < 1, NoID = True];
    S = Union[S, {L}];
    ];
   If[Length[J] > Length[S], NoID = True];
   Return[! NoID];
   ];
\end{verbatim}
\item Function outputting $s_r(K;(x,y))$.
\begin{verbatim}
CodeShare[K_, r_, x_, y_] :=
  Module[{B = rBallSquare[r, x, y], i, RShare = 0},
   For[i = 1, i < Length[B] + 1, i++,
    RShare =
      RShare +
       1/Length[
         Intersection[rBallSquare[r, B[[i]][[1]], B[[i]][[2]]], K]];
    ];
   Return[RShare];
   ];
\end{verbatim}
\item Function outputting an upper approximation of $s_r(K;(x,y))$ (by Lemma 2.1) given a radius r and a code K.
\begin{verbatim}
ApproximatedShare[K_, r_, x_, y_] :=
  Module[{B = rBallSquare[r, x, y], i, j, ISets = {},
    DifferentISets = {}, AShare = 0},
   For[i = 1, i < Length[B] + 1, i++,
    ISets =
      Append[ISets,
       Intersection[rBallSquare[r, B[[i]][[1]], B[[i]][[2]]], K]];
    ];
   DifferentISets = Union[ISets];
   For[j = 1, j < Length[DifferentISets] + 1, j++,
    AShare =
      AShare +
       1/Length[
         DifferentISets[[j]]] + (Count[ISets, DifferentISets[[j]]] -
          1)/(Length[DifferentISets[[j]]] + 1);
    ];
   Return[AShare];
   ];
\end{verbatim}
\end{itemize}

The program codes of the functions for shifting share are defined as follows. Given a code $K$ and a codeword $(x,y)$, the functions output the amount of share that is shifted from $(x,y)$ to other codewords.
\begin{verbatim}
ShiftingRule1[K_, x_, y_] :=
  Module[{B = rBallKing[1, x, y], i, k, RShare = 0},
   If[Intersection[K, B] != {{x, y}}, Return[0]];
   If[! MemberQ[K, {x + 2, y}] && ! MemberQ[K, {x + 2, y - 1}] &&
     MemberQ[K, {x + 3, y}] && MemberQ[K, {x + 3, y - 1}],
    RShare = RShare + 1/5;
    ];
   If[! MemberQ[K, {x + 2, y}] && ! MemberQ[K, {x + 2, y + 1}] &&
     MemberQ[K, {x + 3, y}] && MemberQ[K, {x + 3, y + 1}],
    RShare = RShare + 1/5;
    ];
   If[! MemberQ[K, {x - 2, y}] && ! MemberQ[K, {x - 2, y - 1}] &&
     MemberQ[K, {x - 3, y}] && MemberQ[K, {x - 3, y - 1}],
    RShare = RShare + 1/5;
    ];
   If[! MemberQ[K, {x - 2, y}] && ! MemberQ[K, {x - 2, y + 1}] &&
     MemberQ[K, {x - 3, y}] && MemberQ[K, {x - 3, y + 1}],
    RShare = RShare + 1/5;
    ];
   If[! MemberQ[K, {x, y + 2}] && ! MemberQ[K, {x - 1, y + 2}] &&
     MemberQ[K, {x, y + 3}] && MemberQ[K, {x - 1, y + 3}],
    RShare = RShare + 1/5;
    ];
   If[! MemberQ[K, {x, y + 2}] && ! MemberQ[K, {x + 1, y + 2}] &&
     MemberQ[K, {x, y + 3}] && MemberQ[K, {x + 1, y + 3}],
    RShare = RShare + 1/5;
    ];
   If[! MemberQ[K, {x, y - 2}] && ! MemberQ[K, {x - 1, y - 2}] &&
     MemberQ[K, {x, y - 3}] && MemberQ[K, {x - 1, y - 3}],
    RShare = RShare + 1/5;
    ];
   If[! MemberQ[K, {x, y - 2}] && ! MemberQ[K, {x + 1, y - 2}] &&
     MemberQ[K, {x, y - 3}] && MemberQ[K, {x + 1, y - 3}],
    RShare = RShare + 1/5;
    ];
   Return[RShare];
   ];

ShiftingRule2[K_, x_, y_] :=
  Module[{B = rBallKing[1, x, y], i, k, RShare = 0},
   If[Intersection[K, B] != {{x, y}}, Return[0]];
   If[! MemberQ[K, {x + 2, y}] && MemberQ[K, {x + 3, y}] &&
     MemberQ[K, {x + 4, y}],
    RShare = RShare + 1/30;
    ];
   If[! MemberQ[K, {x - 2, y}] && MemberQ[K, {x - 3, y}] &&
     MemberQ[K, {x - 4, y}],
    RShare = RShare + 1/30;
    ];
   If[! MemberQ[K, {x, y + 2}] && MemberQ[K, {x, y + 3}] &&
     MemberQ[K, {x, y + 4}],
    RShare = RShare + 1/30;
    ];
   If[! MemberQ[K, {x, y - 2}] && MemberQ[K, {x, y - 3}] &&
     MemberQ[K, {x, y - 4}],
    RShare = RShare + 1/30;
    ];
   Return[RShare];
   ];

ShiftingRule3[K_, x_, y_] :=
  Module[{B = rBallSquare[2, x, y], i, k, RShare = 0},
   If[Intersection[K, B] != {{x, y}}, Return[0]];
   If[MemberQ[K, {x + 2, y + 1}] && MemberQ[K, {x + 3, y + 1}],
    RShare = RShare + 1/12];
   If[MemberQ[K, {x + 2, y - 1}] && MemberQ[K, {x + 3, y - 1}],
    RShare = RShare + 1/12];
   If[MemberQ[K, {x - 2, y + 1}] && MemberQ[K, {x - 3, y + 1}],
    RShare = RShare + 1/12];
   If[MemberQ[K, {x - 2, y - 1}] && MemberQ[K, {x - 3, y - 1}],
    RShare = RShare + 1/12];
   If[MemberQ[K, {x + 1, y + 2}] && MemberQ[K, {x + 1, y + 3}],
    RShare = RShare + 1/12];
   If[MemberQ[K, {x - 1, y + 2}] && MemberQ[K, {x - 1, y + 3}],
    RShare = RShare + 1/12];
   If[MemberQ[K, {x + 1, y - 2}] && MemberQ[K, {x + 1, y - 3}],
    RShare = RShare + 1/12];
   If[MemberQ[K, {x - 1, y - 2}] && MemberQ[K, {x - 1, y - 3}],
    RShare = RShare + 1/12];
   Return[RShare];
   ];

ShiftingRule4[K_, x_, y_] :=
  Module[{B = rBallSquare[2, x, y], i, k, RShare = 0},
   If[Intersection[K, B] != {{x, y}}, Return[0]];
   If[MemberQ[K, {x + 2, y + 1}] && MemberQ[K, {x + 2, y + 2}],
    RShare = RShare + 7/60];
   If[MemberQ[K, {x + 2, y - 1}] && MemberQ[K, {x + 2, y - 2}],
    RShare = RShare + 7/60];
   If[MemberQ[K, {x - 2, y + 1}] && MemberQ[K, {x - 2, y + 2}],
    RShare = RShare + 7/60];
   If[MemberQ[K, {x - 2, y - 1}] && MemberQ[K, {x - 2, y - 2}],
    RShare = RShare + 7/60];
   If[MemberQ[K, {x + 1, y + 2}] && MemberQ[K, {x + 2, y + 2}],
    RShare = RShare + 7/60];
   If[MemberQ[K, {x - 1, y + 2}] && MemberQ[K, {x - 2, y + 2}],
    RShare = RShare + 7/60];
   If[MemberQ[K, {x + 1, y - 2}] && MemberQ[K, {x + 2, y - 2}],
    RShare = RShare + 7/60];
   If[MemberQ[K, {x - 1, y - 2}] && MemberQ[K, {x - 2, y - 2}],
    RShare = RShare + 7/60];
   Return[RShare];
   ];

ShiftingRule5[K_, x_, y_] :=
  Module[{B = rBallKing[1, x, y],
    S = Complement[rBallSquare[2, x, y], rBallKing[1, x, y]], i, k,
    RShare = 0},
   If[(Intersection[K, B] != {{x, y}}) || (Length[
        Intersection[S, K]] > 1), Return[0]];
   If[MemberQ[K, {x + 2, y}] && MemberQ[K, {x + 2, y + 1}],
    RShare = RShare + 1/30];
   If[MemberQ[K, {x + 2, y}] && MemberQ[K, {x + 2, y - 1}],
    RShare = RShare + 1/30];
   If[MemberQ[K, {x - 2, y}] && MemberQ[K, {x - 2, y + 1}],
    RShare = RShare + 1/30];
   If[MemberQ[K, {x - 2, y}] && MemberQ[K, {x - 2, y - 1}],
    RShare = RShare + 1/30];
   If[MemberQ[K, {x, y + 2}] && MemberQ[K, {x + 1, y + 2}],
    RShare = RShare + 1/30];
   If[MemberQ[K, {x, y + 2}] && MemberQ[K, {x - 1, y + 2}],
    RShare = RShare + 1/30];
   If[MemberQ[K, {x, y - 2}] && MemberQ[K, {x + 1, y - 2}],
    RShare = RShare + 1/30];
   If[MemberQ[K, {x, y - 2}] && MemberQ[K, {x - 1, y - 2}],
    RShare = RShare + 1/30];
   Return[RShare];
   ];

ShiftingRule6[K_, x_, y_] :=
  Module[{B = rBallKing[1, x, y],
    S = Complement[rBallSquare[2, x, y], rBallKing[1, x, y]], i, k,
    RShare = 0},
   If[(Intersection[K, B] != {{x, y}}) || (Length[
        Intersection[S, K]] > 1), Return[0]];
   If[MemberQ[K, {x + 2, y - 1}] && MemberQ[K, {x + 2, y}] &&
     MemberQ[K, {x + 2, y + 1}], RShare = RShare + 1/20];
   If[MemberQ[K, {x - 2, y - 1}] && MemberQ[K, {x - 2, y}] &&
     MemberQ[K, {x - 2, y + 1}], RShare = RShare + 1/20];
   If[MemberQ[K, {x - 1, y + 2}] && MemberQ[K, {x, y + 2}] &&
     MemberQ[K, {x + 1, y + 2}], RShare = RShare + 1/20];
   If[MemberQ[K, {x - 1, y - 2}] && MemberQ[K, {x, y - 2}] &&
     MemberQ[K, {x + 1, y - 2}], RShare = RShare + 1/20];
   Return[RShare];
   ];

ShiftingRule7[K_, x_, y_] :=
  Module[{B = rBallSquare[2, x, y], i, k, RShare = 0},
   If[Intersection[K, B] != {{x, y}}, Return[0]];
   If[MemberQ[K, {x + 3, y}],
    If[Intersection[rBallSquare[1, x + 2, y + 1],
       K] == {{x + 2, y + 1}}, RShare = RShare + 7/60];
    If[Intersection[rBallSquare[1, x + 2, y - 1],
       K] == {{x + 2, y - 1}}, RShare = RShare + 7/60];
    ];
   If[MemberQ[K, {x - 3, y}],
    If[Intersection[rBallSquare[1, x - 2, y + 1],
       K] == {{x - 2, y + 1}}, RShare = RShare + 7/60];
    If[Intersection[rBallSquare[1, x - 2, y - 1],
       K] == {{x - 2, y - 1}}, RShare = RShare + 7/60];
    ];
   If[MemberQ[K, {x, y + 3}],
    If[Intersection[rBallSquare[1, x + 1, y + 2],
       K] == {{x + 1, y + 2}}, RShare = RShare + 7/60];
    If[Intersection[rBallSquare[1, x - 1, y + 2],
       K] == {{x - 1, y + 2}}, RShare = RShare + 7/60];
    ];
   If[MemberQ[K, {x, y - 3}],
    If[Intersection[rBallSquare[1, x + 1, y - 2],
       K] == {{x + 1, y - 2}}, RShare = RShare + 7/60];
    If[Intersection[rBallSquare[1, x - 1, y - 2],
       K] == {{x - 1, y - 2}}, RShare = RShare + 7/60];
    ];
   Return[RShare];
   ];

ShiftingRule8[K_, x_, y_] :=
  Module[{B = rBallSquare[2, x, y], i, k, RShare = 0},
   If[Intersection[K, B] != {{x, y}}, Return[0]];
   If[! MemberQ[K, {x + 2, y - 1}] && !
      MemberQ[K, {x + 3, y - 1}] && ! MemberQ[K, {x + 2, y + 1}] && !
      MemberQ[K, {x + 3, y + 1}] && ! MemberQ[K, {x + 4, y}] && !
      MemberQ[K, {x + 1, y - 2}] && ! MemberQ[K, {x + 1, y + 2}] &&
     MemberQ[K, {x + 2, y - 2}] && MemberQ[K, {x + 2, y + 2}] &&
     MemberQ[K, {x + 3, y}],
    RShare = RShare + 1/20;
    ];
   If[! MemberQ[K, {x - 2, y - 1}] && !
      MemberQ[K, {x - 3, y - 1}] && ! MemberQ[K, {x - 2, y + 1}] && !
      MemberQ[K, {x - 3, y + 1}] && ! MemberQ[K, {x - 4, y}] && !
      MemberQ[K, {x - 1, y - 2}] && ! MemberQ[K, {x - 1, y + 2}] &&
     MemberQ[K, {x - 2, y - 2}] && MemberQ[K, {x - 2, y + 2}] &&
     MemberQ[K, {x - 3, y}],
    RShare = RShare + 1/20;
    ];
   If[! MemberQ[K, {x - 1, y + 2}] && !
      MemberQ[K, {x - 1, y + 3}] && ! MemberQ[K, {x + 1, y + 2}] && !
      MemberQ[K, {x + 1, y + 3}] && ! MemberQ[K, {x, y + 4}] && !
      MemberQ[K, {x - 2, y + 1}] && ! MemberQ[K, {x + 2, y + 1}] &&
     MemberQ[K, {x - 2, y + 2}] && MemberQ[K, {x + 2, y + 2}] &&
     MemberQ[K, {x, y + 3}],
    RShare = RShare + 1/20;
    ];
   If[! MemberQ[K, {x - 1, y - 2}] && !
      MemberQ[K, {x - 1, y - 3}] && ! MemberQ[K, {x + 1, y - 2}] && !
      MemberQ[K, {x + 1, y - 3}] && ! MemberQ[K, {x, y - 4}] && !
      MemberQ[K, {x - 2, y - 1}] && ! MemberQ[K, {x + 2, y - 1}] &&
     MemberQ[K, {x - 2, y - 2}] && MemberQ[K, {x + 2, y - 2}] &&
     MemberQ[K, {x, y - 3}],
    RShare = RShare + 1/20;
    ];
   Return[RShare];
   ];

ShiftingRule9[K_, x_, y_] :=
  Module[{B = rBallKing[1, x, y],
    S = Complement[rBallSquare[2, x, y], rBallKing[1, x, y]], i, k,
    RShare = 0},
   If[(Intersection[K, B] != {{x, y}}) || (Length[
        Intersection[S, K]] > 1), Return[0]];
   If[MemberQ[K, {x + 2, y}] && MemberQ[K, {x + 1, y - 2}] &&
     MemberQ[K, {x + 1, y + 2}] && ! MemberQ[K, {x + 2, y - 1}] && !
      MemberQ[K, {x + 2, y + 1}],
    RShare = RShare + 7/60;
    ];
   If[MemberQ[K, {x - 2, y}] && MemberQ[K, {x - 1, y - 2}] &&
     MemberQ[K, {x - 1, y + 2}] && ! MemberQ[K, {x - 2, y - 1}] && !
      MemberQ[K, {x - 2, y + 1}],
    RShare = RShare + 7/60;
    ];
   If[MemberQ[K, {x, y + 2}] && MemberQ[K, {x - 2, y + 1}] &&
     MemberQ[K, {x + 2, y + 1}] && ! MemberQ[K, {x - 1, y + 2}] && !
      MemberQ[K, {x + 1, y + 2}],
    RShare = RShare + 7/60;
    ];
   If[MemberQ[K, {x, y - 2}] && MemberQ[K, {x - 2, y - 1}] &&
     MemberQ[K, {x + 2, y - 1}] && ! MemberQ[K, {x - 1, y - 2}] && !
      MemberQ[K, {x + 1, y - 2}],
    RShare = RShare + 7/60;
    ];
   Return[RShare];
   ];

ShiftingRule10[K_, x_, y_] :=
  Module[{B = rBallSquare[2, x, y], i, k, RShare = 0},
   If[Intersection[K, B] != {{x, y}}, Return[0]];
   If[! MemberQ[K, {x + 3, y - 1}] && ! MemberQ[K, {x + 3, y}] && !
      MemberQ[K, {x + 3, y + 1}] && MemberQ[K, {x + 2, y - 1}] &&
     MemberQ[K, {x + 2, y + 1}] && ! MemberQ[K, {x + 2, y - 2}] && !
      MemberQ[K, {x + 2, y + 2}] && MemberQ[K, {x + 4, y}],
    RShare = RShare + 7/60;
    ];
   If[! MemberQ[K, {x - 3, y - 1}] && ! MemberQ[K, {x - 3, y}] && !
      MemberQ[K, {x - 3, y + 1}] && MemberQ[K, {x - 2, y - 1}] &&
     MemberQ[K, {x - 2, y + 1}] && ! MemberQ[K, {x - 2, y - 2}] && !
      MemberQ[K, {x - 2, y + 2}] && MemberQ[K, {x - 4, y}],
    RShare = RShare + 7/60;
    ];
   If[! MemberQ[K, {x - 1, y + 3}] && ! MemberQ[K, {x, y + 3}] && !
      MemberQ[K, {x + 1, y + 3}] && MemberQ[K, {x - 1, y + 2}] &&
     MemberQ[K, {x + 1, y + 2}] && ! MemberQ[K, {x - 2, y + 2}] && !
      MemberQ[K, {x + 2, y + 2}] && MemberQ[K, {x, y + 4}],
    RShare = RShare + 7/60;
    ];
   If[! MemberQ[K, {x - 1, y - 3}] && ! MemberQ[K, {x, y - 3}] && !
      MemberQ[K, {x + 1, y - 3}] && MemberQ[K, {x - 1, y - 2}] &&
     MemberQ[K, {x + 1, y - 2}] && ! MemberQ[K, {x - 2, y - 2}] && !
      MemberQ[K, {x + 2, y - 2}] && MemberQ[K, {x, y - 4}],
    RShare = RShare + 7/60;
    ];
   Return[RShare];
   ];
\end{verbatim}

Assume that $I_2((0,0))= \{(0,0)\}$. Consider sets $D' \subseteq B_3((0,0)) \setminus B_2((0,0))$ such that the set $D = D' \cup \{(0,0)\}$ $2$-separates all the pairs of vertices in $B_1((0,0))$ and that the approximated share (by Lemma 2.1) is greater than 35/6. There exist $209$ such sets $D$ and these sets are listed in \verb"Problems1".
\begin{verbatim}
Problems1 = {};
TestSpace = rBallSquare[1, 0, 0];
SearchSpace = Complement[rBallSquare[3, 0, 0], rBallSquare[2, 0, 0]];
For[i = 1, i < 2^(Length[SearchSpace]) + 1, i++,
  ProposedSet = Union[Subsets[SearchSpace, All, {i}][[1]], {{0, 0}}];
  If[IDonSquareGrid[ProposedSet, TestSpace,
     2] && (ApproximatedShare[ProposedSet, 2, 0, 0] > 35/6),
   Problems1 = Union[Problems1, {ProposedSet}];
   ];
  ];
\end{verbatim}
In what follows, for each set $D = D' \cup \{(0,0)\}$ in \verb"Problems1" inducing problems, we check that for all the subsets $F' \subseteq B_4((0,0)) \setminus B_3((0,0))$ such that $F' \cup D$ 2-separates all the pairs of vertices in $B_2((0,0))$, we have that the share is at most 35/6 after the rules have been applied. Thus, in total, we have to go through $209 \cdot 2^{16}$ different cases.
\begin{verbatim}
G = CreateGrid[10];
GridG = ListPlot[G, AspectRatio -> 1];
TestSpace = rBallSquare[2, 0, 0];
SearchSpace = Complement[rBallSquare[4, 0, 0], rBallSquare[3, 0, 0]];
For[i = 1, i < Length[Problems1] + 1, i++,
  MaxShiftedShare = 0;
  For[j = 1, j < 2^(Length[SearchSpace]) + 1, j++,
   ProposedSet =
    Union[Subsets[SearchSpace, All, {j}][[1]], Problems1[[i]]];
   If[IDonSquareGrid[ProposedSet, TestSpace, 2],
    TempShare =
     CodeShare[ProposedSet, 2, 0, 0] -
      ShiftingRule1[ProposedSet, 0, 0] -
      ShiftingRule2[ProposedSet, 0, 0] -
      ShiftingRule3[ProposedSet, 0, 0] -
      ShiftingRule4[ProposedSet, 0, 0] -
      ShiftingRule5[ProposedSet, 0, 0] -
      ShiftingRule6[ProposedSet, 0, 0] -
      ShiftingRule7[ProposedSet, 0, 0] -
      ShiftingRule8[ProposedSet, 0, 0] -
      ShiftingRule9[ProposedSet, 0, 0] -
      ShiftingRule10[ProposedSet, 0, 0];
    MaxShiftedShare = Max[MaxShiftedShare, TempShare];
    ];
   ];
  Print["Problem set D = ", Problems1[[i]],
   " Maximum share after shifting for D = ", MaxShiftedShare, " ~ ",
   N[MaxShiftedShare]];
  ];
\end{verbatim}

Assume that $I_2((0,0))= \{(0,0),(2,0)\}$. Consider sets $D' \subseteq B_3((0,0)) \setminus B_2((0,0))$ such that the set $D = D' \cup \{(0,0),(2,0)\}$ $2$-separates all the pairs of vertices in $B_1((0,0))$ and that the approximated share (by Lemma 2.1) is greater than 35/6. There exist $35$ such sets $D$ and these sets are listed in \verb"Problems2".
\begin{verbatim}
Problems2 = {};
TestSpace = rBallSquare[1, 0, 0];
SearchSpace = Complement[rBallSquare[3, 0, 0], rBallSquare[2, 0, 0]];
For[i = 1, i < 2^(Length[SearchSpace]) + 1, i++,
  ProposedSet =
   Union[Subsets[SearchSpace, All, {i}][[1]], {{0, 0}, {2, 0}}];
  If[IDonSquareGrid[ProposedSet, TestSpace,
     2] && (ApproximatedShare[ProposedSet, 2, 0, 0] > 35/6),
   Problems2 = Union[Problems2, {ProposedSet}];
   ];
  ];
\end{verbatim}
In what follows, for each set $D = D' \cup \{(0,0),(2,0)\}$ in \verb"Problems2" inducing problems, we check that for all the subsets $F' \subseteq B_4((0,0)) \setminus B_3((0,0))$ such that $F' \cup D$ 2-separates all the pairs of vertices in $B_2((0,0))$, we have that the share is at most 35/6 after the rules have been applied. Thus, in total, we have to go through $35 \cdot 2^{16}$ different cases.
\begin{verbatim}
G = CreateGrid[10];
GridG = ListPlot[G, AspectRatio -> 1];
TestSpace = rBallSquare[2, 0, 0];
SearchSpace = Complement[rBallSquare[4, 0, 0], rBallSquare[3, 0, 0]];
For[i = 1, i < Length[Problems2] + 1, i++,
  MaxShiftedShare = 0;
  For[j = 1, j < 2^(Length[SearchSpace]) + 1, j++,
   ProposedSet =
    Union[Subsets[SearchSpace, All, {j}][[1]], Problems2[[i]]];
   If[IDonSquareGrid[ProposedSet, TestSpace, 2],
    TempShare =
     CodeShare[ProposedSet, 2, 0, 0] -
      ShiftingRule1[ProposedSet, 0, 0] -
      ShiftingRule2[ProposedSet, 0, 0] -
      ShiftingRule3[ProposedSet, 0, 0] -
      ShiftingRule4[ProposedSet, 0, 0] -
      ShiftingRule5[ProposedSet, 0, 0] -
      ShiftingRule6[ProposedSet, 0, 0] -
      ShiftingRule7[ProposedSet, 0, 0] -
      ShiftingRule8[ProposedSet, 0, 0] -
      ShiftingRule9[ProposedSet, 0, 0] -
      ShiftingRule10[ProposedSet, 0, 0];
    MaxShiftedShare = Max[MaxShiftedShare, TempShare];
    ];
   ];
  Print["Problem set D = ", Problems1[[i]],
   " Maximum share after shifting for D = ", MaxShiftedShare, " ~ ",
   N[MaxShiftedShare]];
  ];
\end{verbatim}

\bibliographystyle{abbrv}

\end{document}